\documentclass[11pt]{article}
\usepackage[margin=1in]{geometry}
\usepackage{amssymb, amsthm, amsmath, mathrsfs, amsfonts}
\usepackage{tikz-cd}
\usepackage{mathtools}
\usepackage{theoremref}
\usepackage{pdfpages}
\usepackage[hidelinks]{hyperref}
\usepackage{comment}
\usepackage{dsfont}
\usepackage{graphicx}
\usepackage{subcaption}
\usepackage{appendix}
\usepackage[misc]{ifsym}
\usepackage{enumitem}
\usepackage{authblk} % for author affiliations
\DeclareMathOperator*{\argmax}{arg\,max}
\DeclareMathOperator*{\argmin}{arg\,min}

\newcommand{\T}{\mathcal{T}_c} % Optimal transport cost
\newcommand{\chfun}{\mathds{1}} % characteristic function of a set (1 and 0)
\newcommand{\pf}{F} % primal functional
\newcommand{\df}{G} % dual functional
\newcommand{\kf}{K} % Kantorovich dual functional
\newcommand{\sfct}{H} % Sub-dual functional
\newcommand{\cts}[1]{\mathcal{C}(#1)}
\newcommand{\ctsbdd}[1]{\mathcal{C}_b(#1)}
\newcommand{\base}{B}
\newcommand{\X}{X}
\newcommand{\Y}{Y}
\newcommand{\svar}{{x}}
\newcommand{\shz}{s}
\newcommand{\shza}{\tilde{s}}
\newcommand{\svt}{p}

\newcommand{\tvar}{y}
\newcommand{\tvara}{\tilde{y}}
\newcommand{\ul}{u} % dummy variable for upper limit of integration
\newcommand{\ula}{\tilde{u}} % alternative dummy variable for upper limit of integration
\newcommand{\bm}{{\bar{p}}}      % base measure
\newcommand{\bmalt}{{q}}      % alternative measure on base space
\newcommand{\im}{{\mu}} % induced measure
\newcommand{\tm}{{\nu}}
\newcommand{\leb}[1]{\mathcal{L}^{#1}}
\newcommand{\n}{n}

\newcommand{\ext}[1]{#1_{\mathrm{ext}}} % extended sets for reformulation as SDOT
\newcommand{\cst}{\tau} % constant that makes surface integrate to 1
\newcommand{\ev}{\sigma} % evaluation point of inverse cost
\newcommand{\PV}{Q} % Ertel PV as a function
\newcommand{\PVv}{q} % Ertel PV value
\newcommand{\PT}{\Theta} % potential temperature as a function
\newcommand{\PTv}{\theta} % potential temperature value
\newcommand{\ZAM}{Z} % zonal angular momentum as a function
\newcommand{\ZAMv}{z} % zonal angular momentum value
\newcommand{\cp}{c_{\mathrm{p}}} % specific heat capacity at constant pressure
\newtheorem{thm}{Theorem}[section]
\newtheorem{lem}[thm]{Lemma}
\newtheorem{defn}[thm]{Definition}
\newtheorem{prop}[thm]{Proposition}
\newtheorem{cor}[thm]{Corollary}
\newtheorem{rem}[thm]{Remark}
\newtheorem{ass}[thm]{Assumption}
\newtheorem{problem}[thm]{Problem}
\newcommand{\R}{\mathbb{R}}
\newcommand{\N}{\mathbb{N}}
\newcommand{\pdone}[2]{\dfrac{\partial {#1}}{\partial {#2}}}

\newcommand{\PM}{\mathcal{P}} % Probability measures
\newcommand{\Pac}{\mathcal{P}_{\mathrm{ac}}}
\newcommand{\rd}{\mathrm{d}}
\newcommand{\eps}{\varepsilon}
\newcommand{\spt}{\mathrm{spt}}
\newcommand{\mres}{%
	\,\raisebox{-.127ex}{\reflectbox{\rotatebox[origin=br]{-90}{$\lnot$}}}\,%
}
\newcommand{\edit}[1]{#1}

\begin{document}

\title{A definition of the background state of the atmosphere\\using optimal transport}
\date{}
\author{Charlie Egan\footnote{\edit{Department of Meteorology, University of Reading, UK; c.egan@reading.ac.uk}}, John Methven\footnote{Department of Meteorology, University of Reading, UK; j.methven@reading.ac.uk}, David P.~Bourne\footnote{David P. Bourne, Maxwell Institute for Mathematical Sciences and Department of
Mathematics, Heriot-Watt University, Edinburgh, UK; d.bourne@hw.ac.uk}, Mike J.~P.~Cullen\footnote{Met Office, Exeter, UK (retired)}}

\maketitle

\begin{abstract}
The dynamics of atmospheric disturbances are often described in terms of displacements of air parcels relative to their locations in a notional background state.
Modified Lagrangian Mean (MLM) states have been proposed by M.~E.~McIntyre using the Lagrangian conserved variables potential vorticity and potential temperature to label air \edit{masses in the instantaneous 3-D atmospheric state, and then rearranging air masses sandwiched between potential vorticity contours in potential temperature layers so that they are zonally symmetric}. Methven and Berrisford \edit{used conservation of Kelvin's circulation to define the distribution of MLM state mass in zonal angular momentum ($z$) and potential temperature ($\theta$) coordinates given 3-D global data.}
\edit{We prove that for any such mass distribution in $(z, \theta)$ there exists a mapping to a unique mass distribution in geophysical coordinates (latitude and pressure) which minimises energy and completely defines the MLM state in a single hemisphere.}
Since the state is an energy minimum, it is suitable for quantification of finite amplitude disturbances and examining atmospheric instability.
This state is obtained by solving a free surface problem, which we frame as the minimisation of an optimal transport cost over a class of source measures.
The solution consists of a source measure, encoding surface pressure, and an optimal transport map, connecting the distribution of mass in geophysical coordinates to the known distribution of mass in $(z, \theta)$. 
We show that this problem reduces to an optimal transport problem with a known source measure, which has a numerically feasible discretisation, \edit{and we present a numerical solution using atmospheric reanalysis data.}
Additionally, our results hold for a large class of cost functions, and generalise analogous results on free surface variants of the semi-geostrophic equations.
\end{abstract}

\section{Introduction}\label{sect:intro}

The dynamics of atmospheric disturbances and geophysical fluid instability are often described in terms of displacements of air parcels relative to a notional background state. Such disturbances include large amplitude waves, vortices, and other weather systems. In \cite{mcintyre:80a}, Modified Lagrangian Mean (MLM) states were proposed as a suitable choice of background state by using the conserved variables potential vorticity ($\PVv$) and potential temperature ($\PTv$) to label air parcels, thus avoiding the need to calculate 3-D trajectories explicitly to measure disturbance activity. In a chaotic dynamical system, neighbouring trajectories separate exponentially on average, rendering them sensitive to initial conditions and increasingly complex and infeasible to calculate, even numerically, while $\PVv$ and $\PTv$ can typically be calculated from instantaneous data. \edit{The MLM state is then defined as an adiabatic rearrangement of the 3-D state so that potential vorticity contours on each $\theta$-isosurface are zonally symmetric, reducing dimensions to a latitudinal-vertical cross-section. The authors of \cite{methven2015slowly} obtained a zonally-symmetric MLM state numerically from global atmospheric (3D) data by calculating mass and circulation integrals as functionals of $\PVv$ and $\PTv$. Since these integrals are conserved for adiabatic, frictionless flows, the MLM state is obtained as a zonally-symmetric flow with the same integrals as the 3-D state which is sufficient to specify the distribution of mass in the conserved variable coordinates defined by zonal angular momentum ($\ZAMv$) and $\PTv$. In the current work, we prove the existence, uniqueness and stability (with respect to input data) of the MLM state in geophysical coordinates (latitude and pressure) with the minimum energy, given this distribution of mass in conserved variable coordinates $(\ZAMv, \PTv)$. In this way, the MLM state is uniquely defined from the 3-D atmospheric state, as the energy-minimising zonally-symmetric adiabatic rearrangement of fluid parcels. We achieve this by adapting the framework set out in \cite[Section 4.5]{cullen2021mathematics} without restricting our analysis to incompressible flows.} This provides a rigorous principle by which to select a unique background state of the atmosphere suitable for the examination of fluid dynamical instability.

We phrase the problem of finding energy-minimising MLM states as the minimisation of an optimal transport cost over a suitable space of source measures ($\mu$) for a given target measure ($\nu$). In the language of optimal transport, the source space corresponds to geophysical coordinates (latitude and pressure) and the target space is defined by the conserved variable coordinates $(\ZAMv, \PTv)$. In physical terms, the optimal transport solution maps the mass from the target space (distribution taken as given) into the source space, while also minimising a cost which is shown to be equal to the energy of the state.

We now describe the general mathematical problem that we study, our main results, and their interpretation in the context of finding energy minimising MLM states. For more detail on this interpretation see Section \ref{sect:relating} and Figure \ref{fig:MLM_schematic}. Let $d\in\N$, $d>1$, and let $\base\subset \R^{d-1}$ be compact. Define the source space $\X\coloneqq \base\times [0,+\infty)$. For a probability density $\bm$ defined on $\base$, define the set
\[
\X_\bm\coloneqq \{(\shz,\svt) \in \X \, : \, 0\leq \svt\leq \bm(\shz)\}.
\]
That is, $\X_\bm$ is the intersection of the subgraph of $\bm$ with the upper half-space.
Let $\mu_\bm$ be the restriction of the Lebesgue measure to $\X_\bm$, and let $\tm$ be a probability measure on $\R^d$ with support contained in a compact set $\Y\subset \R^d$. For a cost function $c:\R^d\times\R^d\to\R$ denote by $\T(\im_\bm,\tm)$ the optimal transport cost from $\im_\bm$ to $\tm$ for the cost $c$.
We study the minimisation problem
\begin{equation}\label{eqn:intro_inf}
\argmin_{\bm\in \Pac(\base)} \T(\im_\bm,\tm).
\end{equation}
Here, $\Pac(\base)$ is the space of probability measures on $\base$ that are absolutely continuous with respect to the $(d-1)$-dimensional Lebesgue measure, and we conflate a measure in $\Pac(\base)$ with its density.

\begin{figure}
    \centering
    \includegraphics[width=0.75\textwidth,trim={3cm 6.2cm 7.2cm 6cm},clip]{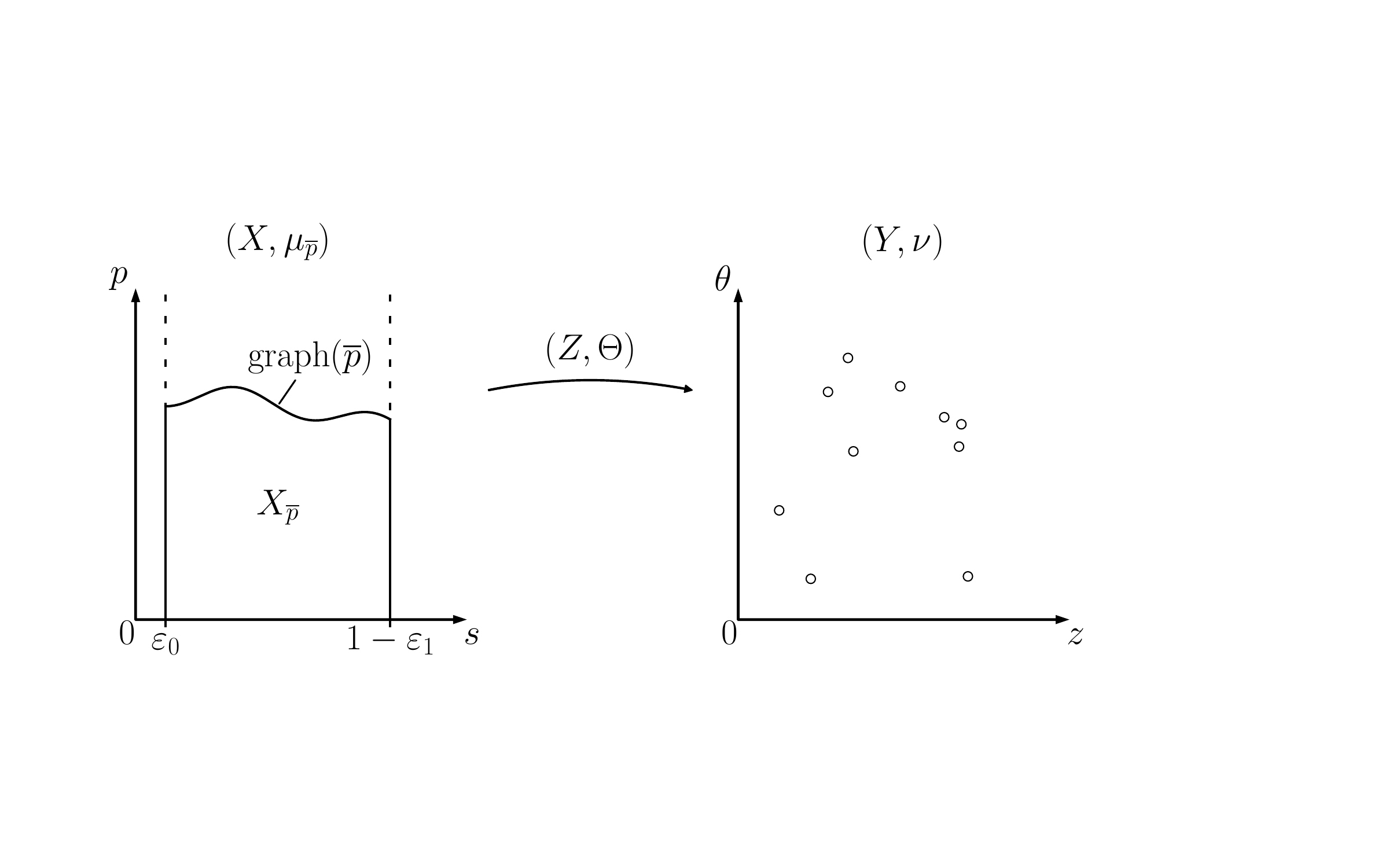}
    \caption{Schematic diagram of an MLM state at an instant in time. The measure $\tm$ represents the distribution of mass in the space of zonal angular momentum $\ZAMv$ and potential temperature $\PTv$. It is assumed to have support contained in a compact set $Y\subset (0,+\infty)^2$. The source space $\X=[\eps_0,1-\eps_1]\times[0,+\infty)$ has coordinates $s=\sin(\phi)$, where $\phi$ is latitude, and pressure $p$. The constants $\eps_0,\, \eps_1>0$ exclude the pole and the equator. The surface pressure $\overline{p}$ is a function of $s$ and is an unknown of the problem. It determines the measure $\im_{\bm}$ as the restriction of the Lebesgue measure to the set $X_{\overline{p}}$. Zonal angular momentum $Z$ and potential temperature $\Theta$ are unknown scalar functions of $(s,p)$. The triple $(\bm,\ZAM,\PT)$ is an MLM state for $\tm$ if $(\ZAM,\PT)$ is an admissible transport map from $\im_{\bm}$ to $\tm$.}
    \label{fig:MLM_schematic}
\end{figure}

In the setting of MLM states, $d=2$, the coordinates $\shz$ and $\svt$ of the source space $\X$ represent sine-of-latitude and pressure, respectively, and the set $B$ is a compact interval. The measure $\tm$ represents the mass distribution over the space of zonal angular momentum and potential temperature of the atmosphere at an \edit{instant} in time. An MLM state is determined by a latitude-dependent probability density $\bm$, representing zonally-symmetric surface pressure, and an admissible transport map from $\im_\bm$ to $\tm$, representing an zonally-symmetric mass-preserving rearrangement of the zonal angular momentum and potential temperature of the full state. From a physical perspective, the optimal transport map relates air parcels and their conserved properties, given by the parcel locations in $(\ZAMv, \PTv)$, to their distribution in geophysical coordinates. The cost function $c$ defined by \eqref{eqn:bgs_cost} represents the \edit{sum of kinetic energy and enthalpy}, and the corresponding transport cost is the energy of the MLM state. Assumptions \ref{ass:cost} and \ref{ass:ext_twist} are both satisfied by $c$, so our results establish a rigorous principle uniquely selecting MLM states based on energy minimisation.

Our main result is that for a large class of cost functions $c$ (Assumption \ref{ass:cost}), the minimisation problem \eqref{eqn:intro_inf} has a unique solution $\bm$, which has continuous density (Theorem \ref{thm:duality_existence_uniqueness}), and can be recovered by solving an optimal transport problem with a known source measure, an adjusted target measure and an extended cost function (see Theorem \ref{thm:reduction} and Figure \ref{fig:reduction_schematic}). We use this to derive a numerically tractable discretisation of \eqref{eqn:intro_inf}; see Corollary \ref{cor:reduction_semi_discrete}. Moreover, for target measures $\nu$ with common compact support, we show that the set of all such solutions is uniformly bounded and equicontinous, and that solutions are stable with respect to $\nu$ (Theorem \ref{thm:min_conv}). Our existence, uniqueness and stability results generalise analogous results obtained for cost functions defining free surface variants of the semi-geostrophic equations \cite{cheng2016semigeostrophic, cullen2001variational,cullen2019stability} and a model for axisymmetric vortices \cite{cullen2014model}, which rely on the use of convex potentials and their relation to optimal transport with the quadratic cost.

\begin{figure}
    \centering
    \includegraphics[width=0.75\textwidth,trim={3cm 6.2cm 7.2cm 6cm},clip]{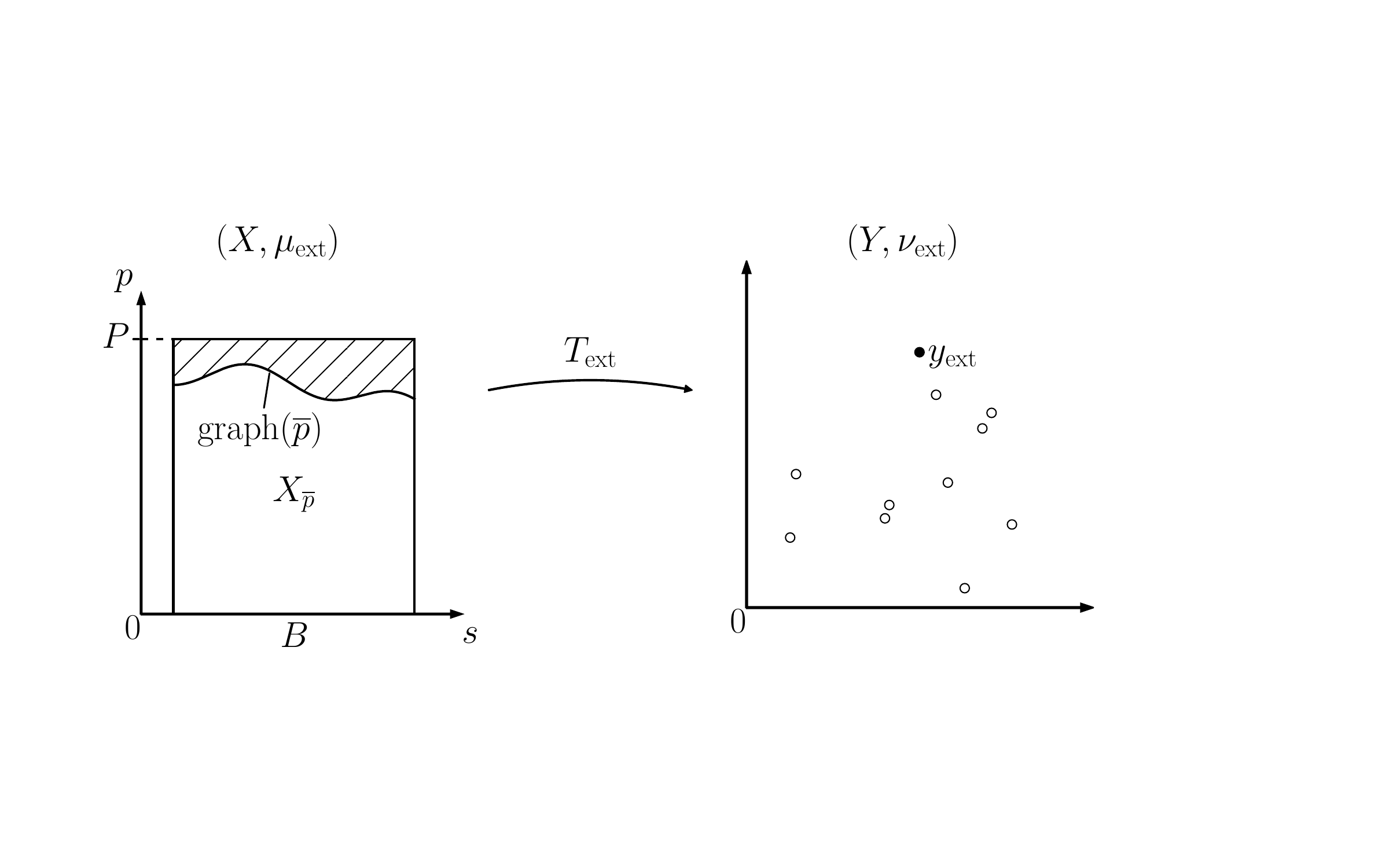}
    \caption{Reduction of Problem \ref{prob:primal} to the optimal transport problem between the `extended' source measure $\ext{\im} = \leb{d}\mres (\base\times [0,P])$ and the `extended' target measure $\ext{\tm}=\nu + (P-1)\delta_{\ext{\tvar}}$. Here, $\ext{\im}=\leb{d}\mres (\base\times [0,P])$, where $\base$ is a compact subset of $\R^{d-1}$ and the constant $P>0$ is an upper bound on the minimiser ${\bm}$ of Problem \ref{prob:primal}. The measure $\tm$ is represented by empty circles. The point $\ext{\tvar}$ is represented by a filled circle, and is outside the support of $\tm$. The cost function is extended to $\ext{\tvar}$ by zero. The optimal transport map $\ext{T}$ sends the hatched cell to $\ext{y}$, the lower boundary of this cell is the graph of $\bm$, and the restriction of $\ext{T}$ to $\X_\bm$ is the optimal transport map from $\im_\bm$ to $\tm$.}
    \label{fig:reduction_schematic}
\end{figure}

\subsection{Motivation and physical problem statement}

Considering atmospheric disturbances as perturbations from a background state, wave activity conservation laws, which describe how disturbance amplitude can propagate from one location to another, have been obtained for arbitrary large amplitude disturbances provided that the flow is adiabatic and frictionless and the background state is itself a solution of the equations of motion \cite{mcintyre:shepherd}. The laws are derived using Lagrangian conservation properties of fluid dynamics to label air masses. For global atmospheric dynamics on the sphere, as modelled by the primitive equations, these labels are potential temperature ($\PTv$) and Ertel potential vorticity ($\PVv$).

Modified Lagrangian Mean (MLM) states \cite{mcintyre:80a} are background states of the atmosphere defined in terms of $\PVv$ and $\PTv$ labels. Every $\PVv$-contour in an MLM state is zonally symmetric and the state is a solution of the equations of motion. Atmospheric motions can be described with reference to an MLM state by displacement of $\PVv$-contours relative to their reference positions. Since MLM states are zonally symmetric and time symmetric (steady) for adiabatic motions, there are two disturbance conservation laws for pseudomomentum and pseudoenergy \cite{haynes:88,methven:13}.
This follows generally from Noether's theorem relating conservation laws to symmetries in Hamiltonian dynamical systems \cite{shepherd:90}. However, the definition of the disturbances themselves depends on the choice of MLM state. This is not satisfactory if you would like to answer questions like whether the atmosphere is unstable and disturbances are expected to grow, so a principle by which to uniquely select an MLM state is desirable.

Several heuristics have been used to define and compute MLM states.
In \cite{mcintyre:80a}, an MLM state is defined by an adiabatic rearrangement of fluid parcels, with $\PTv$ increasing with height and zonally-symmetric $\PVv$ increasing polewards on surfaces of constant $\PTv$. 
\edit{
In \cite{methven2015slowly}, an MLM state with (approximately) the same mass distribution in coordinates $(\PVv,\PTv)$ as the full atmospheric flow is defined using two time-invariant integral quantities: mass and Kelvin's circulation. Consider an instantaneous 3-dimensional state of the atmosphere expressed in \emph{isentropic coordinates}, where the three independent variables are sine-of-latitude $s$, longitude $\lambda$ and potential temperature $\PTv$. Let $\PV$ be a function of $(s,\lambda,\theta)$ representing the potential vorticity of the 3-D state. For $\PVv$-levels $\PVv_l>0$ and $\PTv$-levels $\PTv_m>0$ on an axis-aligned rectangular grid in $\R^2$ with half-level depths $\delta \PVv>0$ and $\delta \PTv>0$, mass and circulation integrals are computed over each volume
\[
\mathcal{V}(\PVv_l,\PTv_m)\coloneqq 
\big\{
(s,\lambda,\theta)\in [0,1]\times [0,2\pi)\times \R
\; \big\vert \;
\PVv_l - \delta \PVv\leq \PV(s,\lambda,\PTv),\;
\PTv \in [\PTv_l-\delta\PTv,\,\PTv_l+\delta\PTv]
\big\}.
\]
That is, $\mathcal{V}(\PVv_l,\PTv_m)$ is the $3$-dimensional volume of air bounded laterally by the $\PVv$ contour at $\PVv_l-\delta\PVv$ and vertically by the $\PTv$-surfaces at $\PTv_l-\delta\PTv$ and $\PTv_l+\delta\PTv$. The corresponding mass and circulation integrals are
\begin{align}
    \label{eqn:mass_int}
    &\mathcal{M}(\PVv_l,\PTv_m)\coloneqq \iiint_{\mathcal{V}(\PVv_l,\PTv_m)} r(s,\lambda,\PTv)a^2 \, \rd s \rd \lambda \rd \PTv,\\
    \label{eqn:circ_int}
    &\mathcal{C}(\PVv_l,\PTv_m)\coloneqq \iiint_{\mathcal{V}(\PVv_l,\PTv_m)} \zeta(s,\lambda,\PTv) a^2 \, \rd s \rd \lambda \rd \PTv,
\end{align}
where $r$ is the isentropic mass density and $\zeta=r \PV$ is the vertical component of absolute isentropic vorticity, both of the 3-dimensional state, and $a$ is the Earth's radius. A zonally-symmetric MLM state is defined by requiring that its mass and circulation integrals are the same as those computed from the instantaneous 3-dimensional state. In this way, the MLM state is described as an adiabatic rearrangement of the full state because if the flow were adiabatic and frictionless then the full state could be accessible from the background state through adiabatic motions of the fluid, which conserve $\PVv$, $\PTv$, $\mathcal{M}$ and $\mathcal{C}$.
}

\edit{
In this paper, we rephrase the constraint that the mass and circulation integrals of the MLM state are the same as those of the full 3-D state as a constraint on the mass density of the MLM state in coordinates $(\ZAMv,\PTv)$, where $\ZAMv$ is zonal angular momentum. 
We then minimise the integral of kinetic energy and enthalpy subject to this constraint to obtain an energy-minimising MLM state. 
While $\ZAMv$ is not a conserved variable of the 3-D state, under the condition of zonal symmetry, it is possible to change variables from $(\PVv,\PTv)$ to $(\ZAMv,\PTv)$. Indeed, for a zonally symmetric state, $\mathcal{C}$ is proportional to $\ZAMv$:
\begin{equation}
\label{eqn:circ}
    \mathcal{C}(\PVv,\PTv)=2\pi \ZAMv.
\end{equation}
Restricting to positive values of $\PVv$ (in the Northern Hemisphere) means that the vertical component of absolute vorticity $\zeta=r\PVv$ is everywhere positive, so $\mathcal{C}$ is strictly decreasing in $\PVv$ on every $\PTv$-level and \eqref{eqn:circ} defines a change of variables from $(\PVv,\PTv)$ to $(\ZAMv,\PTv)$. In particular, the mass density of the 3-D state in $(\PVv,\PTv)$ can be pushed forward under this change of variables to coordinates $(\ZAMv,\PTv)$ and normalised to obtain a normalised mass density (or, more generally, a probability measure) $\nu$. The constraint on the MLM state is satisfied precisely when its normalised mass density in coordinates $(\ZAMv,\PTv)$ is equal to $\nu$.
}

In order to calculate the full properties of a background state, and corresponding perturbations, it is necessary and sufficient to obtain its zonal angular momentum and potential temperature as functions of latitude $\phi$ and pressure $p$\edit{, as well as the surface pressure $\bm$ as a function of latitude. The function $\bm$ defines the extent of the physical domain in pressure coordinates and represents the pressure at the lower boundary of the atmosphere. Let $\ZAM$ denote the MLM zonal angular momentum as a function of $s=\sin(\phi)$ and $p$, and let $\PT$ denote the MLM potential temperature as a function of $s$ and $p$.}
\edit{By \eqref{eqn:circ}, $\ZAM$ is constrained to take values $\ZAMv = \mathcal{C}(\PVv,\PTv)/2\pi$, where $\PVv$ and $\PTv$ are values of potential  vorticity and potential temperature, respectively, labelling air masses. A zonally-symmetric adiabatic rearrangement of the 3-D circulation specifies the coordinates $(s,p)$ to which these values of $\ZAMv$ are assigned, thus defining the function $\ZAM$. Since circulation is regarded as a functional of potential vorticity and potential temperature, this also defines $\Theta$ and the MLM state potential vorticity as functions of $(s,p)$. To obtain the MLM state zonal wind (relative to the Earth's surface) $u$, one may then use the relation
\begin{equation}
\label{eqn:Z}
    \ZAMv = ua \cos\phi+\Omega a^2 \cos^2\phi,
\end{equation}
where $\Omega$ is the planetary rotation rate.
Likewise, the MLM state temperature $T$ may be recovered via the relation
\begin{equation}
\label{eqn:theta}
    \PTv = T \left( \frac{p}{p_r} \right)^{-\kappa},
\end{equation}
where $p_r$ is a constant reference pressure and $\kappa=2/7$ is the Poisson constant for a diatomic ideal gas.
}

\edit{Obtaining the background variables $\ZAM$ and $\PT$ as functions of $s$ and $p$, and $\bm$ as a function of $s$, is essential to the analysis of perturbation evolution and atmospheric stability, and it is particularly desirable to obtain an MLM state that is an energy minimiser, accounting consistently for the surface pressure. With these variables, it is possible to find any quantities defined as perturbations of $\PVv$-contours relative to their positions in the background state (called the ``equivalent latitudes" of $\PVv$-contours), as well as perturbation wind and energy. Moreover, finite-amplitude stability results for the atmosphere, which depend on knowing $\overline{p}$ \cite{bowman:shepherd:95}, yield upper bounds on disturbance energy.}

\subsection{Optimal transport}

We recall some fundamental concepts and results from optimal transport theory that form the basis of our analysis of MLM states. A more thorough exposition of optimal transport theory can be found in several standard texts such as \cite{merigot2021optimal}, \cite{santambrogio2015optimal} or \cite{villani2008optimal}.

The starting point of optimal transport is the Monge problem: given $\X,\,\Y\subset \R^d$, a cost function $c:\X\times\Y\to\R$, Borel probability measures $\mu\in\PM(\X)$ and $\nu\in\PM(\Y)$, find
\begin{equation}\label{eqn:Monge}
\inf\left\{\left.\int_{\X} c(\svar,T(\svar))\, \rd \mu(\svar)\, \right\vert\, T:\X\to\Y \text{ is Borel measureable and }\, T_\#\mu = \nu\right\}.
\end{equation}
Here, $T_\#\mu$ is the probability measure on $\Y$ defined by
\[
(T_\#\mu)(A) = \mu(T^{-1}(A))
\]
for all Borel sets $A\subseteq \Y$,
which is called the pushforward of $\mu$ by $T$. The spaces $X$ and $Y$ are often referred to as the \emph{source} and \emph{target} spaces, respectively, and the measures $\mu$ and $\nu$ as the source and target measures.
If $T$ satisfies the constraint $T_\#\mu = \nu$, then it is said to be an \emph{admissible transport map} from $\mu$ to $\nu$. If $T$ attains the infimum in \eqref{eqn:Monge} it is said to be an optimal transport map from $\mu$ to $\nu$. 

Admissible transport maps do not necessarily exist. For example, if $\mu$ is a single Dirac mass and $\nu$ is a sum of two weighted Dirac masses at distinct locations,
then
the Monge problem is infeasible. The standard relaxation of the Monge problem is the Kantorovich problem, in which the optimal transport cost is defined by
\begin{equation}\label{eqn:OT_cost}
\T(\mu,\nu)\coloneqq \underset{\gamma\in \Gamma(\mu,\nu)}{\inf}\left\{\int_{\X} c(\svar,\tvar)\, \rd \gamma(\svar,\tvar)\right\},
\end{equation}
where
\[
\Gamma(\mu,\nu)\coloneqq \{\gamma\in  \PM(\X\times\Y)\, \vert\, {\pi_\X}_\# \gamma = \mu,
\;
{\pi_\Y}_\# \gamma = \nu\},
\]
and where $\pi_\X$ and $\pi_\Y$ denote the projections onto $\X$ and $\Y$, respectively.
The convex set $\Gamma(\mu,\nu)$ is called the set of
\emph{admissible transport plans} from $\mu$ to $\nu$. Under mild assumptions on the cost function $c$, the Kantorovich problem has a solution $\gamma$, which is called an optimal transport plan from $\mu$ to $\nu$, and $\T(\mu,\nu)$ is the optimal transport cost from $\mu$ to $\nu$. The Kantorovich Duality Theorem states that
\begin{equation}\label{eqn:kant_duality}
\T(\mu,\nu) = \sup_{\psi\in \cts{\Y}}\left\{\int_{\X}\psi^c(\svar)\,\rd\mu(\svar) + \int_\Y \psi(\tvar)\, \rd\nu(\tvar)\right\},
\end{equation}
where $\cts{\Y}$ is the space of continuous functions on $\Y$, and $\psi^c:\X\to\R \cup \{-\infty\}$ is the $c$-transform of $\psi$ defined by
\begin{equation}\label{eqn:c_transform}
    \psi^c(\svar) = \inf_{\tvar\in\Y} \{c(\svar,\tvar) - \psi(\tvar)\}.
\end{equation}
(See, e.g., \cite[Theorem 1.39]{santambrogio2015optimal} or \cite[Theorem 5.10]{villani2008optimal}.)
Again, under mild assumptions on the cost function $c$
and the sets $\X$ and $\Y$, the supremum in \eqref{eqn:kant_duality} is attained. A function $\psi$ that achieves the maximum is called a \emph{Kantorovich potential}.

The Gangbo-McCann Theorem gives conditions under which the Monge problem has a unique solution. (See, for example, \cite{gangbo1996geometry}, \cite[Theorem 12]{merigot2021optimal}, or \cite[Theorem 10.28]{villani2008optimal}.) Moreover, it ties together the three formulations of optimal transport given above. Indeed, if $T$ is the optimal transport map, then $\gamma\coloneqq (\mathrm{id}_\X,T)_\# \mu$ is the unique optimal transport plan, and for any Kantorovich potential $\psi$,
\begin{equation}\label{eqn:psi_gamma_opt}
c(\svar,\tvar)= \psi^c(\svar) + \psi(\tvar)\quad \text{for }\gamma\text{-almost-every }(\svar,\tvar)\in \X\times\Y.
\end{equation}
In particular, the optimal transport map can be recovered from a Kantorovich potential 
via the expression $T(\svar)=(\nabla_x c(\svar,\cdot))^{-1}(\nabla \psi^c(\svar))$. This expression is well defined when $\nabla_x c(\svar,\cdot)$ exists and is injective, a condition known as the \emph{twist} condition.

\subsection{Relating Modified Lagrangian Mean states to optimal transport}\label{sect:relating}

To formulate the definition of an MLM state mathematically, consider \edit{again a instantaneous 3-dimensional atmospheric state}. Let $\nu$ be a probability measure on $\R^2$ representing the normalised mass distribution of this state in zonal angular momentum and potential temperature coordinates \edit{(obtained from the 3-D atmospheric state as explained in the introduction)}. We consider states with positive zonal angular momentum. Since potential temperature is positive, and both zonal angular momentum and potential temperature are bounded, $\nu$ has compact support contained in $(0,+\infty)^2$. Let $Y\subset (0,+\infty)^2$ be a compact set containing the support of $\nu$. Let $s=\sin(\phi)$, where $\phi$ is latitude. Let $p_{\mathrm{min}}>0$ be a prescribed and arbitrary minimum pressure, and let $p$ be the positive deviation of pressure from $p_{\mathrm{min}}$. We seek zonally-symmetric MLM states defined in pressure coordinates $(s,p)$ on a single hemisphere with a spherical cap from the pole and a small region close to the equator removed. That is, for some $\eps_{0},\,\eps_{1}\in (0,1/2)$ , we consider the physical domain
\[
X\coloneqq [\eps_0,1-\eps_1]\times [0,+\infty)
\]
containing points $\svar=(s,p)$. Under the assumption of hydrostatic balance (i.e., that the fluid density multiplied by gravitational acceleration $g$, is the negative vertical gradient of pressure) the fluid density in pressure coordinates $(s,p)$, often called the pseudodensity, is uniform (and equal to $-1/g$) \cite[Chapter 6]{hoskins2014fluid}. An MLM state is then defined by three maps: surface pressure $\overline{p}$, potential temperature $\PT$, and zonal angular momentum $\ZAM$. In the following definition $\Pac([\eps_0,1-\eps_1])$ denotes the space of probability measures on $[\eps_0,1-\eps_1]$ that are absolutely continuous with respect to the Lebesgue measure, and we conflate such measures with their densities.

\begin{defn}[MLM state, c.f. Figure \ref{fig:MLM_schematic}]\label{def:MLM_state}
Let $\overline{p}\in \Pac([\eps_0,1-\eps_1])$.
Define the domain
\[
\X_{\overline{p}} \coloneqq \{(s,p)\in \X\, : \, p\leq \overline{p}(s)\}.
\]
Consider Borel measurable maps $\ZAM,\,\PT: \X_{\overline{p}}\to \R$, with extension by zero to $X$.
The triple $(\overline{p},\ZAM,\PT)$ is an MLM state for $\nu$ if and only if
\begin{equation}\label{eqn:push_forward}
(\ZAM,\PT)_\#\im_{\overline{p}} = \nu,
\end{equation}
where
\[
\im_{\overline{p}}\coloneqq\leb{d}\mres \X_{\overline{p}}
\]
is the restriction of the $d$-dimensional Lebesgue measure to the set $\X_{\overline{p}}$.
\end{defn}

Interpreted physically, the pushforward constraint \eqref{eqn:push_forward} means that $(\overline{p},\ZAM,\PT)$ represents a \edit{zonally-symmetric} mass-preserving rearrangement of the \edit{instantaneous circulation and potential temperature of the 3-D state}. Mathematically, it means that $(\ZAM,\PT)$ is an \emph{admissible transport map} from $\im_{\overline{p}}$ to $\tm$. 

\edit{We now define a cost function that represents the sum of kinetic energy and enthalpy of a zonally-symmetric state, and we interpret the corresponding transport cost as the energy of this state.} Our goal of finding energy-minimising MLM states can then be phrased as minimising the optimal transport cost to $\tm$ over all source measures of the form $\im_{\overline{p}}$ for $\overline{p}:
[\eps_0,1-\eps_1]
\to[0,+\infty)$.

\begin{defn}[Background state cost function]\label{def:bgs_energy}
Define the cost function $c:\X \times \Y\to [0,+\infty)$ by
\begin{equation}\label{eqn:bgs_cost}
c((s,p),(\ZAMv,\PTv))\coloneqq \frac{1}{2}\left(\frac{\ZAMv}{a\sqrt{1-s^2}} - \Omega a\sqrt{1-s^2}\right)^2
+\cp\PTv\left(\frac{p+p_{\mathrm{min}}}{p_r}\right)^\kappa,
\end{equation}
where $a>0$ is the radius of the earth, $\Omega>0$ is the angular velocity of the earth, $\cp>0$ is the specific heat capacity of air at constant pressure, $p_{\mathrm{min}}>0$ is a prescribed and arbitrary minimum pressure, $p_\mathrm{r}>0$ is a constant reference pressure, and \edit{$\kappa=R/\cp$, where $R$ is the gas constant so that $\kappa=2/7$}.
\end{defn}

\begin{rem}
Using $p_{\mathrm{min}}>0$, we force pressure to be positive, which ensures that the cost function is locally Lipschitz. This is not restrictive when computing MLM states because pressure tending zero corresponds to leaving the atmosphere and entering outer space, so atmospheric data will only ever extend to small but non-zero pressure. However, $p=0$ would be the natural physical boundary condition, so we hope to consider the case $p_{\mathrm{min}}=0$ in future work.
\end{rem}

\edit{Using the relations \eqref{eqn:Z} and \eqref{eqn:theta}, we see that
\[
c = \frac{1}{2}u^2 + \cp T,
\]
which is the sum of kinetic energy and enthalpy. The energy $E$ of an MLM state $(\overline{p},Z,\Theta)$ is given by the integral of \edit{kinetic energy and enthalpy} over the whole hemispheric domain:}
\begin{equation}\label{eqn:bgs_energy}
E(\overline{p},\ZAM,\PT) = \int_{\eps_0}^{1-\eps_1}\int_{0}^{\overline{p}(s)} c\big((s,p),(\ZAM(s,p),\PT(s,p))\big)\, \rd p \,\rd s.
\end{equation}
It follows by definition of $\im_{\overline{p}}$ that
\[
E(\overline{p},\ZAM,\PT) = \int_{\X} c\big(\svar,(\ZAM,\PT)(\svar)\big)\, \rd\im_{\overline{p}}(\svar).
\]
This is precisely the cost of transporting the measure $\im_{\overline{p}}$ to the measure $\nu$ using the admissible transport map $(\ZAM,\PT)$ with the cost function $c$. For a given surface pressure $\overline{p}$, the minimal energy over all $(\ZAM,\PT)$ such that $(\overline{p},\ZAM,\PT)$ is an MLM state is $\T(\im_{\overline{p}},\tm)$. To minimise the energy $E$ over \emph{all} MLM states $(\overline{p},\ZAM,\PT)$, we therefore minimise $\T(\im_{\overline{p}},\tm)$ over all surface pressures $\overline{p}\in \Pac([\eps_0,1-\eps_1])$.

\begin{problem}[MLM state energy minimisation]\label{prob:primal_informal}
Given a distribution $\nu \in \mathcal{P}(Y)$ of zonal angular momentum and potential temperature, find a corresponding MLM state $(\overline{p}_*,\ZAM_*,\PT_*)$ with minimal energy. That is, find
\[
\overline{p}_* \in \argmin_{\overline{p}\in \Pac([\eps_0,1-\eps_1])}\T(\im_{\overline{p}},\nu),
\]
and
\[
(\ZAM_*,\PT_*)\in \argmin_{(\ZAM,\PT)_\# \im_{\overline{p}_*}=\tm} \int_{\X} c\big(\svar,(\ZAM,\PT)(\svar)\big)\, \rd\im_{\overline{p}_*}(\svar).
\]
\end{problem}

We prove existence, uniqueness and stability of solutions of a general formulation of Problem \ref{prob:primal_informal}, defined in Problem \ref{prob:primal}, under assumptions on the cost function and the source and target spaces as set out in Section \ref{sect:prob}. In Proposition \ref{prop:bgs_cost} we show that the cost function \eqref{eqn:bgs_cost} defining the energy of an MLM state satisfies these assumptions, meaning that the results that we prove are applicable to Problem \ref{prob:primal_informal}.

\subsection{Rearrangements and optimal transport in the atmospheric sciences}

The current work builds on the long-standing connection between optimal transport and the atmospheric sciences \cite{cullen2021mathematics}. The cornerstone of this connection is the interpretation of energy-minimising mass-preserving rearrangements of vector fields as optimal transport maps. We use this perspective to interpret the definition of MLM states in terms of optimal transport, which enables us to prove their existence, uniqueness and stability. Meteorologists Cullen and Purser first used \emph{monotone} rearrangements of vector fields in their study of the \emph{semi-geostrophic} (SG) equations \cite{cullen1984extended}, which model the formation of atmospheric fronts, and found them to be minimisers of the \emph{geostrophic energy}. Brenier's celebrated Polar Factorisation Theorem \cite{brenier1991polar} later identified such rearrangements as optimal transport maps for the quadratic cost. Interpreting the geostrophic energy as a transport cost and the energy-minimising rearrangement as an optimal transport map subsequently allowed for the rigorous mathematical analysis of the SG equations in several settings (see \cite{cullen2021mathematics} for a comprehensive overview).

The 3-dimensional free surface \cite{cheng2016semigeostrophic, cullen2019stability}, 2-dimensional shallow water \cite{cullen2001variational}, and 3-dimensional compressible \cite{cullen2003fully,faria2013existence,cullen2014solutions} variants of the SG equations, and the model for forced axisymmetric flows considered in \cite{cullen2014model}, are all continuity equations with velocity defined at each time instant by minimising a transport cost over a space of source measures, similar to Problem \ref{prob:primal_informal}.
In the 3-dimensional free surface SG equations, given a target measure $\tm\in\PM(\R^3)$, the energy functional is defined for $\bm\in\Pac(\R^2)$, modelling surface pressure, by 
\[
\bm \mapsto \T(\mu_\bm,\nu).
\]
Here $\mu_\bm$ is the restriction of the $3$-dimensional Lebesgue measure to the subgraph of $\bm$ intersected with the upper half space, as in Definition \ref{def:MLM_state}, and the cost function is given by
\begin{equation}\label{eqn:SG_cost}
c(x,y) = \frac{1}{2}|x_1 - y_1|^2 + \frac{1}{2}|x_2 - y_2|^2 + x_3 y_3.
\end{equation}
\edit{For the 3-dimensional compressible SG equations in pressure coordinates \cite[Section 3.3.6]{cullen2021mathematics}, the energy functional takes the same form but with the term $x_3y_3$ in the cost function replaced by $x_3^{\kappa}y_3$, up to a positive multiplicative constant.
Both corresponding energy minimisation problems are special cases of Problem \ref{prob:primal}.
The vertical coordinate of the target space represents potential temperature, so the support of the target measure $\nu$ is contained in the upper half-space. It is straightforward to verify that the cost function \eqref{eqn:SG_cost} satisfies Assumption \ref{ass:cost} and the twist condition (Assumption \ref{ass:twist}).
A similar argument as used in Proposition \ref{prop:bgs_cost} verifies that these assumptions are also satisfied by the cost function of the 3-D compressible SG equations in pressure coordinates.}

The 2-dimensional shallow water SG equations are the special case of the $3$-dimensional free surface SG equations where potential temperature is constant. The energy functional then takes the form
\begin{equation}\label{eqn:CSG_energy}
    \rho \mapsto \T(\rho,\nu) + \int \rho^\gamma
\end{equation}
where $\rho$ represents the fluid density, $\nu$ is a given target measure, $c$ is a given transport cost, and $\gamma=2$. This energy minimisation problem is the same as computing the Moreau envelope in the Wasserstein space, as studied in \cite{sarrazin2022lagrangian}. The energy functional for the compressible SG equations is also of the form \eqref{eqn:CSG_energy} but with $\gamma \in (1,2)$.

\edit{Outside of applications in the atmospheric sciences, there are clear parallels between our approach to defining and computing MLM states and the use of optimal transport in the problem of reconstructing the early Universe. There, the aim is to compute a mapping between mass distributions representing the early and late Universe. The cost function is quadratic and the transport cost represents the integrated kinetic energy, while the Kantorovich potential $\psi$ and its $c$-transform relate to the final and initial gravitational potentials. Not only has optimal transport been used to prove the existence of a unique reconstruction under the condition that the fluid velocity is the gradient of a convex potential \cite{frisch2002reconstruction,brenier2003reconstruction}, it also provides the basis of a numerical method. In \cite{LevyMohayaeevonHausegger2021}, a semi-discrete optimal transport algorithm is used to compute a mapping between a uniform distribution of mass, representing the early Universe, onto a discrete mass distribution, where each point represents a galaxy in the late Universe. The efficiency of this method means that it can be applied at scale, enabling the reconstruction of detailed features and making it suitable for practical purposes.}

\subsection{Towards numerical computation of MLM states}\label{sect:intro_numerics}

The MLM state in \cite{methven2015slowly} was computed using a numerical technique called equivalent latitude iteration with potential vorticity inversion (ELIPVI). While numerical solutions were obtained, questions of existence and uniqueness of such solutions, as well as convergence of the numerical scheme, were not addressed. It was also not considered whether or not the computed background states were energy minimising.

We show in Section \ref{sect:reduct} that the problem of finding energy minimising MLM states (Problem \ref{prob:primal_informal}) reduces to an optimal transport problem where a single weighted Dirac mass has been added to the target measure and the cost to that point is set to zero, as illustrated in Figure \ref{fig:reduction_schematic}. This is similar in spirit to the notion of partial optimal transport considered in \cite{levy2022partial} and used to generate Lagrangian meshes with free boundaries.
With a view to solving Problem \ref{prob:primal_informal} numerically, in Corollary \ref{cor:reduction_semi_discrete} we consider the case where the mass distribution $\nu$ is a discrete measure. This corresponds precisely to the setting considered in \cite{methven2015slowly}. \edit{There, mass and circulation integrals $\mathcal{M}(\PVv_l,\PTv_m)$ and $\mathcal{C}(\PVv_l,\PTv_m)$ are computed for the full 3-dimensional atmospheric state for points $(\PVv_l,\PTv_m)$ on an axis-aligned rectangular grid in $\R^2$ \edit{with half-level depths $\delta \PVv$ and $\delta \PTv$ (see equations \eqref{eqn:mass_int} and \eqref{eqn:circ_int}). Correspondingly, $\nu$ would be given by
\begin{equation}\label{eqn:nu_computed}
\nu = \frac{1}{M}\sum_{(l,m)}
\left(\mathcal{M}(\PVv_{l},\PTv_m) - \mathcal{M}(\PVv_{l+1},\PTv_m)\right)
\delta_{\left(\frac{\mathcal{C}(\PVv,\PTv)}{2\pi},\PTv\right)},
\end{equation}
where $M$ is a normalisation constant} and we recall that circulation is proportional to zonal angular momentum for zonally-symmetric flow \eqref{eqn:circ}.} By Corollary \ref{cor:reduction_semi_discrete}, Problem \ref{prob:primal_informal} then reduces to a standard semi-discrete optimal transport problem, which is numerically tractable \cite[Section 4]{merigot2021optimal}. Hence, our analysis not only provides rigorous means by which to select an MLM state given a full atmospheric flow, with guarantees of existence, uniqueness and stability, it also provides a numerical method for computing such states. \edit{As a proof of concept, Figure \ref{fig:computed_tessellation} contains a numerically-computed energy-minimising MLM state obtained from climate reanalysis data using this numerical method. A more comprehensive implementation of this method will be the subject of a separate work.}

\subsection{Notation and conventions}
For $d\in \N = \{1,2,3,\ldots\}$, let $A\subseteq \R^d$ be a Borel set. We denote the interior of $A$ by $\mathrm{int}(A)$, the boundary of $A$ by $\partial A$, the diameter of $A$ by $\mathrm{diam}(A)$, and the characteristic function of $A$ by $\chfun_A$. That is,
\[
\chfun_A(x) = 
\begin{cases}
1 \text{ if } x\in A,\\
0 \text{ otherwise.}
\end{cases}
\]
For a function $f:A\to \R$, we denote by $\mathrm{graph}(f)$ its graph, that is, $\mathrm{graph}(f) \coloneqq \{(x,f(x))\, :\, x\in A\}$.
We denote the Lebesgue measure of dimension $d$ by $\leb{d}$, and its restriction to $A$ by $\leb{d}\mres A$. The space of Borel probability measures on $A$ is denoted by $\PM(A)$. We denote by $\spt(\mu)$ the support of a measure $\mu\in \PM(A)$. The set of Borel probability measures on $A$ that are absolutely continuous with respect to $\leb{d}$ is denoted by $\Pac(A)$. We conflate a measure $\mu\in\Pac(A)$ with its density with respect to $\leb{d}$, and the corresponding element of $L^1(A)$. We denote by $\cts{A}$ and $\ctsbdd{A}$ the spaces of continuous functions on $A$ and bounded continuous functions on $A$, respectively.
When $A$ is a compact set, we equip $\cts{A}$ with the uniform norm, which we denote by $\|\cdot\|_{\cts{A}}$.
Given Borel sets $A,\, B \subseteq \R^d$, a cost function $c:A\times B\to \R$, and measures $\mu\in \PM(A)$ and $\nu\in\PM(B)$, we denote by $\T(\mu,\nu)$ the optimal transport cost from $\mu$ to $\nu$. For a Borel measureable function $f:A\to B$, the pushforward of $\mu$ by $f$ is the measure $f\# \mu\in \PM(B)$ defined by $(f\# \mu)(U) \coloneqq \mu(f^{-1}(U))$ for Borel sets $U\subseteq B$. For $\phi\in\cts{A}$ and $\psi\in\cts{B}$ we denote by $\phi\oplus\psi$ the element of $\cts{A\times B}$ defined by $\left(\phi\oplus\psi\right)(x,y) = \phi(x)\oplus\psi(y)$. For $\mu\in \PM(A)$ and $\nu\in \PM(B)$ we denote by $\mu\otimes \nu$ the product measure defined by $(\mu\otimes\nu)(U\times V)=\mu(U)\nu(V)$ for Borel sets $U\subseteq A$ and $V\subseteq B$.

\subsection{Outline and main contributions}

A precise mathematical statement of the energy minimisation problem, Problem \ref{prob:primal}, studied in this paper is given in Section \ref{sect:prob}. Our main contributions are:
\begin{itemize}
    \item Theorem \ref{thm:duality_existence_uniqueness}, which establishes the existence and uniqueness of solutions via duality, and gives optimality conditions relating the solutions of the primal and dual problems.
    \item Theorem \ref{thm:min_conv}, which establishes stability of solutions with respect to input data.
    \item Theorem \ref{thm:reduction} and Corollaries \ref{cor:reduction_map} and \ref{cor:reduction_semi_discrete}, which reduce the free surface minimisation problem, Problem \ref{prob:primal}, to a standard optimal transport problem, and treat the case of twisted costs and discrete target measures.
\end{itemize}

\section{Problem statement and assumptions}\label{sect:prob}

In this section we precisely state the problem under study in this paper (Problem \ref{prob:primal}), which is a direct generalisation of Problem \ref{prob:primal_informal} for finding energy-minimising MLM states. We also state assumptions under which our results hold (Assumptions \ref{ass:cost}, \ref{ass:twist}, and \ref{ass:ext_twist}), \edit{and show that they} are all satisfied by the cost defining the energy of an MLM state; see Proposition \ref{prop:bgs_cost}.

Let $d\geq 2$, and let $\base\subset\R^{d-1}$ be a compact set with non-empty interior. Define the source domain $\X\coloneqq\base\times [0,+\infty)$, and let $\Y\subset \R^d$ be a non-empty compact target domain. We denote by $\svar = (\shz,\svt)$ a point in $\X$ and by $\tvar$ a point in $\Y$.

\begin{ass}\label{ass:cost}
The cost function $c:\X\times\Y\to [0,+\infty)$ satisfies the following properties:
\begin{enumerate}[label=\Alph*.]
	\item $c$ is locally Lipschitz.\label{ass:cost_cts}
	\item For all $\shz\in\base$ and all $\tvar\in\Y$ the function $c((\shz,\cdot),\tvar):[0,+\infty)\to[0,+\infty)$ is strictly increasing and unbounded.\label{ass:cost_inc}
    \item The set of functions $\{c((\cdot,\svt),\tvar):\base\to[0,+\infty)\, \vert \, \svt\in [0,+\infty),\,\tvar\in\Y\}$ is equicontinuous.\label{ass:cost_unif_cts_x}
\end{enumerate}
\end{ass}

Unless otherwise stated, we consider cost functions $c:\X\times\Y\to[0,+\infty)$ satisfying only Assumption \ref{ass:cost}.

% Definition of primal functional
\begin{defn}[Primal functional]\label{defn:primal}
For ${\bm}\in\Pac(\base)$ define the set
\[
\X_{{\bm}} \coloneqq \{(\shz,\svt)\in\X\, \vert \, \svt\leq \bm(\shz)\}.
\]
Define \edit{the functional}
\begin{align}
\nonumber
&\edit{\im:\Pac(\base)\to\Pac(\X),}\\
\label{eqn:mu_p}
&\edit{\bm\mapsto\im_\bm \coloneqq \leb{d}\mres \X_{{\bm}}.}
\end{align}
For $\tm\in\PM(\Y)$ and a cost function $c:\X\times\Y\to [0,+\infty)$, we define the primal functional 
\begin{align*}
    &\edit{\pf_\tm:\Pac(\base)\to [0,+\infty),}\\
    &\edit{\pf_\tm(\bm)\coloneqq\T(\im_\bm,\tm),}
\end{align*}
where $\T$ is the optimal transport cost from $\im_\bm$ to $\tm$ for the cost function $c$ \eqref{eqn:OT_cost}.
\end{defn}

\begin{rem}\label{rem:well_def}
For each $\bm\in\Pac(\base)$, $X_\bm$ is well defined up to $\leb{d}$-negligible sets, $\chfun_{\X_\bm}$ is a well-defined element of $L^1(\X)$ and is the Lebesgue density of $\im_\bm$, and $\im_\bm$ is a probability measure. Moreover, by definition, for any $\bm_0,\,\bm_1\in \Pac(\base)$, 
\[
\|\mu_{\bm_0} - \mu_{\bm_1}\|_{L^1(\X)} = \|\bm_0 - \bm_1\|_{L^1(\base)},
\]
and for any $\xi\in \ctsbdd{\X}$,
\[
\int_{\X}\xi\, \rd \im_{\bm_0} = \int_{\base}\int_{0}^{\bm_0(\shz)}\xi(\shz,\svt)\,\rd \svt\, \rd \shz.
\]
\end{rem}

%%%%% Definition of primal problem
\begin{problem}[Primal problem]\label{prob:primal}
Given $\tm\in\PM(\Y)$, find
\[
\bm_*\in \argmin_{\bm\in\Pac(\base)}\pf_{\tm}(\bm).
\]
\end{problem}

When Assumption \ref{ass:cost} holds, for any $\tm\in\PM(\Y)$ Problem \ref{prob:primal_informal} has a unique solution (Theorem \ref{thm:duality_existence_uniqueness}) and this solution has a continuous density. The set of all such solutions is uniformly bounded and equicontinuous (Lemma \ref{lem:unif_bdd_equicts}), and Problem \ref{prob:primal_informal} is stable with respect to $\tm$ (Theorem \ref{thm:min_conv}).

\begin{ass}[Twist condition]\label{ass:twist}
There exist open sets $\Omega_\X,\,\Omega_\Y\subset \R^d$ such that $X\subset\Omega_\X$ and $\Y\subset\Omega_\Y$, and an extension of the cost function $c$ onto $\Omega_\X\times\Omega_\Y$ such that $c\in \mathcal{C}^1(\Omega_\X\times\Omega_\Y)$ and for all $\svar_0\in \Omega_\X$ the function $\nabla_{\svar} c(\svar_0,\cdot):\Omega_\Y\to\R^d$ is injective.
\end{ass}

The twist condition is a classical sufficient condition for the existence of transport maps \cite[Definition 8]{merigot2021optimal}. We show that when both Assumptions \ref{ass:cost} and \ref{ass:twist} hold, Problem \ref{prob:primal} admits a solution $\bm$ and that there exists a unique optimal transport map between $\im_\bm$ and $\tm$, and that this map is stable with respect to $\nu$; see Theorem \ref{thm:transport_maps}.

In Section \ref{sect:reduct}, we show that under Assumption \ref{ass:cost}, Problem \ref{prob:primal} can be reduced to a single optimal transport problem by appropriate extension of the source and target measures. For the existence of a corresponding optimal transport map, we impose the following condition on the cost function $c$.

\begin{ass}[Extended twist condition]\label{ass:ext_twist}
Assumption \ref{ass:twist} holds and there exists a constant $\ell>0$ such that $\underset{\tvar\in\Y}{\min}\|\nabla_{\svar} c(\svar_0,\tvar)\|_{\R^d}\geq\ell$ for all $\svar_0\in \Omega_\X$.
\end{ass}
\edit{
We conclude this section by showing that the cost function \eqref{eqn:bgs_cost} defining the background state energy \eqref{eqn:bgs_energy} satisfies Assumptions \ref{ass:cost} and \ref{ass:ext_twist}.}

\begin{prop}\label{prop:bgs_cost}
Let $\eps_0,\, \eps_1\in (0,1/2)$, define $\X=[\eps_0,1-\eps_1]\times [0,+\infty)$, and let $\Y\subset (0,+\infty)^2$ be compact and non-empty.
The cost function $c:\X\times \Y\to [0,+\infty)$ defined by \eqref{eqn:bgs_cost} satisfies assumptions \ref{ass:cost} and \ref{ass:ext_twist}.
\end{prop}

\begin{proof}
First we show that $c$ satisfies Assumption \ref{ass:ext_twist}. Define $\Omega_\X \coloneqq (\eps_0/2,1-\eps_1/2) \times (-p_{\min}/2,+\infty)$. Since $\Y\subset (0,+\infty)^2$ is compact, there exists $\theta_{\min}>0$ such that $\Omega_\Y \coloneqq (0,+\infty) \times(\theta_{\min},+\infty)$ contains $\Y$. Extend $c$ to $\Omega_\X\times\Omega_\Y$ by the formula \eqref{eqn:bgs_cost}. Then $c\in \mathcal{C}^1(\Omega_\X\times\Omega_\Y)$ and for $((s,p),(z,\theta)) \in \Omega_\X\times \Omega_\Y$,
\begin{equation}\label{eqn:c_deriv}
\nabla_{(s,p)}c((s,p),(z,\theta)) =
\left(
\begin{array}{c}
\dfrac{z^2s}{a^2\left(1-s^2\right)^2} - \Omega^2 a^2 s\\
\dfrac{\kappa \cp \theta}{p_r}\left(\dfrac{p+p_{\mathrm{min}}}{p_r}\right)^{\kappa-1}
\end{array}
\right).
\end{equation}
Since $s>0$ and $p+p_{\min}>0$, $\nabla_{(s,p)}c((s,p),\cdot)$ is injective on $\Omega_\Y$. It is immediate from \eqref{eqn:c_deriv} that for all $(s,p)\in \Omega_\X$,
\[
\|\nabla_{(s,p)}c((s,p),\cdot)\|_{\cts{\Omega_\Y}}
\geq
\frac{\kappa \cp \theta_{\min}}{p_r}\left(\frac{p_{\min}}{2p_r}\right)^{\kappa-1}\eqqcolon \ell >0,
\]
so Assumption \ref{ass:ext_twist} holds.

Now we show that $c$ satisfies Assumption \ref{ass:cost}. Since $c\in \mathcal{C}^1(\Omega_\X\times\Omega_\Y)$, it is locally Lipschitz on $\X\times \Y$, so Assumption \ref{ass:cost}.\ref{ass:cost_cts} holds. 
Assumption \ref{ass:cost}.\ref{ass:cost_inc} holds because $\kappa$, $\cp$, $p_{\min}$, $p_r>0$, so the map
\[
p \mapsto \cp\theta\left(\frac{p+p_{\min}}{p_r}\right)^{\kappa}
\]
is strictly increasing and unbounded on $[0,+\infty)$ for all $\theta\in (\theta_{\min},+\infty)$.
Since $\Y$ is compact, there exists a positive constant $z_{\max}$ such that $\Y\subset (0,z_{\max})\times (0,+\infty)$. 
Let $((s,p),(z,\theta))\in \X\times\Y$. Since $s\leq 1-\eps_1$ and $\eps_1\in(0,1/2)$,
\[
a^2\left(1-s^2\right)^2 
\geq a^2\left(1-(1-\eps_1)^2\right)^2
= a^2\eps_1^2(2-\eps_1^2)
\geq a^2\eps_1^2.
\]
It follows from \eqref{eqn:c_deriv} that
\[
\left\vert\pdone{}{s}c((s,p),(z,\theta))\right\vert
\leq
\frac{z_{\max}^2}{a^2\eps_1^2} + \Omega^2a^2\eqqcolon L.
\]
The set 
\[
\left\{c((\cdot,p),(z,\theta)):[\eps_0,1-\eps_1]\to [0,+\infty) \; \big\vert \; p\in [0,+\infty),\, (z,\theta)\in \Y\right\}
\]
therefore consists of $L$-Lipschitz functions so is equicontinuous (Assumption \ref{ass:cost}.\ref{ass:cost_unif_cts_x}). 
\end{proof}

Since the cost function \eqref{eqn:bgs_cost} defining the background state energy \eqref{eqn:bgs_energy} satisfies Assumptions \ref{ass:cost} and \ref{ass:ext_twist}, the existence, uniqueness, and stability of energy-minimising MLM states (i.e., solutions of Problem \ref{prob:primal_informal}), follows from Theorems \ref{thm:duality_existence_uniqueness}, \ref{thm:min_conv}, and \ref{thm:transport_maps}. Moreover, by Corollary \ref{cor:reduction_map}, any such state can be found by solving the Monge problem \eqref{eqn:Monge} with an adjusted target measure and an extended cost function.

\section{Existence and uniqueness of optimal surfaces}

\edit{
In this section, we prove that, for any $\tm\in\PM(\Y)$, Problem \ref{prob:primal} has a unique solution that is characterised by necessary and sufficient optimality conditions. This result is summarised in Theorem \ref{thm:duality_existence_uniqueness}. We work in the setting described in Section \ref{sect:prob} and our working hypothesis is that the cost function $c$ satisfies Assumption \ref{ass:cost}. For brevity, we fix $\nu$ and will often drop subscripts $\tm$ in the notation. We begin by formally deriving a dual maximisation problem for the minimisation Problem \ref{prob:primal}, as well as a formula for optimal surfaces. The crucial step of the analysis is to prove uniform a priori estimates on the boundedness and continuity of these surfaces (Lemma \ref{lem:unif_bdd_equicts}), which enables us to prove the existence of a maximiser in the dual problem (Theorem \ref{thm:dual}). Optimality conditions that we prove for the dual problem then link back to the primal problem and allow us to prove our uniqueness and existence result.
}

\edit{
We now formally derive the dual problem and thereby introduce the functionals that will be analysed in this section. By Kantorovich duality,
\begin{equation}\label{eqn:primal_KD}
\pf_\tm(\bm) = \sup_{\psi\in\cts{\Y}} \left\{ \int_{\base}\int_0^{\bm(s)} \psi^c(s,p)\, \rd p\,\rd s + \int_{\Y} \psi\, \rd \tm\right\}.
\end{equation}
After taking the infimum over $\bm$, the infimum and supremum may be swapped up to inequality to yield weak duality
\begin{equation}\label{eqn:weak_duality}
\inf_{\bm\in\Pac(\base)}\pf_\tm(\bm) \leq 
\sup_{\psi\in\cts{\Y}} \df_\tm(\psi),
\end{equation}
where we define the dual functional
\begin{equation}\label{eqn:dual}
    \begin{aligned}
    &\df_\tm:\cts{Y}\to\R,\\
    &\df_\tm(\psi)=\inf_{\bm\in\Pac(\base)}\left\{\int_{\base}\int_0^{\bm(s)} \psi^c(s,p)\, \rd p\,\rd s + \int_{\Y} \psi\, \rd \tm\right\}.
    \end{aligned}
\end{equation}
Define the function
\begin{align*}
    &f_\psi:\base\times\R\to\R\cup\{+\infty\},\\
    &f_\psi(s,u) = 
    \begin{cases}
        \displaystyle{\int_0^u \psi^c(s,p)\, \rd p} \quad &\text{if }u\geq 0,\\
        +\infty  \quad &\text{otherwise},
    \end{cases}
\end{align*}
and define the \emph{inner-dual} functional by
\begin{equation}\label{eqn:inner-dual}
\begin{aligned}
&\sfct_{\tm,\psi}:\PM(\base)\to\R\cup\{+\infty\},\\
&\sfct_{\tm,\psi}(\bm) =
\begin{cases}
\displaystyle{\int_{\base}f_\psi(s,\bm(s))\, \rd s + \int_\Y \psi\, \rd \tm} \qquad &\text{if}\quad \bm\in\Pac(\base),\\
+\infty\qquad &\text{otherwise}.
\end{cases}
\end{aligned}
\end{equation}
By construction,
\[
\df_\tm(\psi) = \inf_{\bm\in\PM(\base)}\sfct_{\tm,\psi}(\bm).
\]
}

\edit{
Notice that, under the assumptions on the cost function and for each fixed $s\in\base$, the function $f_\psi(s,\cdot)$
is strictly convex on the interval $[0,+\infty)$ and convex, proper and lower semi-continuous. Indeed, the partial derivative of $f_\psi$ with respect to $u$ on the interval $[0,+\infty)$ is $u\mapsto \psi^c(s,u)$. (Here we take the right-derivative at the left endpoint.) This is the pointwise minimum over a set of strictly increasing functions, which is itself strictly increasing.
The properties of $f_\psi$ suggest that the inner-dual functional $\sfct_{\tm,\psi}$ is convex. If this were true, then the primal functional $\pf_\tm$, as the pointwise supremum over a set of convex functions, would itself be convex. On the other hand, the dual functional $\df_\tm$ is defined as a pointwise infimum over a set of Kantorovich dual functionals parametrised by $\bm$, each of which is concave. As such, $\df_\tm$ is concave.
}

\edit{
Given convexity of $\pf_\tm$ and concavity of $\df_\tm$, the general framework of the Fenchel-Rockafellar Duality Theorem could be applied, at least formally, in a way similar to the proof of the Kantorovich Duality Theorem in \cite[Section 6.4.1]{carlier2022classical} to deduce that the inequality in \eqref{eqn:weak_duality} is, in fact, an equality. The corresponding primal-dual optimality conditions, which are necessary and sufficient conditions for $\bm$ and $\psi$ to be optimal in their respective problems, are that $\bm$ minimises the functional $\sfct_{\tm,\psi}$ and $\psi$ is maximal in \eqref{eqn:primal_KD}. The latter of these conditions means that $\psi$ is a Kantorovich potential for source measure $\im_\bm$ and target measure $\tm$, and is similar to the result that the subdifferential of the optimal transport cost with respect to the source measure is the set of Kantorovich potentials; see e.g. \cite[Proposition 7.17]{santambrogio2015optimal}. We see, then, that to rigorously characterise solutions $\bm$ of the primal problem, it is crucial to analyse the inner-dual functional $\sfct_{\tm,\psi}$ and to characterise its minimisers. By doing so, we arrive at a direct proof of duality and optimality that avoids explicit use of the Fenchel-Rockafellar Duality Theorem and the associated functional framework, but we note that this approach could be pursued to yield similar results.
}

\subsection{Minimisers of the inner-dual functional}

In this section, we analyse the inner-dual functional $\sfct_{\psi}$ for $\psi\in\cts{\Y}$, where the subscript $\tm$ is dropped for brevity. A formal calculation suggests that the first variation (in the sense of \cite[Definition 7.12]{santambrogio2015optimal}) of $\sfct_{\psi}$ at $\bm$ is the function
\[
\base\ni \shz \mapsto \pdone{f_\psi}{\ul}(\shz ,\bm(\shz))=\psi^c(\shz,\bm(\shz))\in\R.
\]
Then, if $\bm$ is a minimiser of $\sfct_{\psi}$ with continuous density, the first-order necessary optimality conditions stated in \cite[Theorem 7.20]{santambrogio2015optimal} would imply that there exists a constant $\cst\in\R$ such that
\begin{equation}
\label{eqn:opt_cond_p}
\begin{cases}
\psi^c(\shz,\bm(\shz)) = \cst \quad & \text{if }\bm(\shz)>0,\\
\psi^c(\shz,\bm(\shz)) \geq \cst \quad &\text{if }\bm(\shz)=0.
\end{cases}
\end{equation}
\edit{This is formally equivalent to the condition that the subdifferential of $\sfct_{\psi}$ at $\bm$ contains the zero element.}
By definition of the $c$-transform, if \eqref{eqn:opt_cond_p} holds then for any $\shz\in\base$,
\[
c((\shz,\bm(\shz)),\tvar) - \psi(\tvar) \geq \cst \quad \forall\, \tvar\in\Y.
\]
Since $c$ and $\psi$ are continuous on the compact set $\Y$, if $\bm(\shz)>0$ then this holds with equality for some $\tvar\in\Y$.
By Assumptions \ref{ass:cost}.\ref{ass:cost_cts}-\ref{ass:cost_inc}, $c((\shz,\cdot),\tvar)$ is continuous and strictly increasing on $[0,+\infty)$, so it is invertible and its inverse is strictly increasing and non-negative. It follows that
\[
\bm(\shz) \geq  c((\shz,\cdot),\tvar)^{-1}(\psi(\tvar) + \cst) \quad \forall\, \tvar\in\Y,
\] 
which implies that
\[
\bm(\shz) \geq  \max_{\tvar\in\Y} 
\left\{
c((\shz,\cdot),\tvar)^{-1}(\psi(\tvar) + \cst)
\right\},
\]
with equality if $\bm(\shz)>0$. This motivates the following definition.

\begin{defn}
    Define $\bmalt: \base \times \Y \times \R \to [0,+\infty)$ by
    \begin{equation}\label{eqn:q}
        \bmalt(\shz,\tvar,\ev)\coloneqq 
        \begin{cases}
            c((\shz,\cdot),\tvar)^{-1}(\ev)  \quad &\text{if }\ev \geq c((\shz,0),\tvar),\\
            0 \quad &\text{otheriwse}.
        \end{cases}
    \end{equation}
\end{defn}

To formulate our candidate for the unique minimiser of $\sfct_\psi$, we show in Lemma \ref{lem:bm} that there exists a unique constant $\tau=\tau_\psi$ such that $\bm=\bm_\psi$ defined by inverting the optimality condition \eqref{eqn:opt_cond_p} is a probability density. In Theorem \ref{thm:subdual}, we will show that $\bm_\psi$ is indeed the unique minimiser of $\sfct_{\psi}$. We begin with the following lemma, which will also be used in Lemma \ref{lem:unif_bdd_equicts} to show that the set of all functions $\bm_\psi$ for $\psi\in\cts{\Y}$ is uniformly bounded and equicontinuous.

\begin{lem}\label{lem:Q_unifbdd_equicts}
    For any compact interval $U\subset \R$, the set
    \[
    \mathcal{Q} \coloneqq \{\bmalt(\cdot,\tvar,\cdot)\vert_{\base\times U} \}_{\tvar\in\Y}
    \]
    is uniformly bounded and equicontinuous.
\end{lem}

\begin{proof}
Let $U\subset \R$ be a compact interval. We first prove that $\mathcal{Q}$ is uniformly bounded. 
Note that
\begin{equation}\label{eqn:q_max_min}
    q(s,y,\sigma) = \max \left\{ 0, \, \min \{ t \in \R : c((s,t),y) \ge \sigma \} \right\}.
\end{equation}
Let $(\svt_n)_{n \in \N}$ be a sequence in $\R$ diverging to $+ \infty$, for each $n\in \N$ let
\[
(\shz_n,\tvar_n) \in \argmin \{ c((\shz,\svt_n),\tvar) :\, \shz \in \base, \, \tvar \in \Y \},
\]
and let $(\shz_*,\tvar_*)$ be a cluster point of the sequence $((\shz_n,\tvar_n))_{n\in \N}$, which exists by compactness of $\base$ and $\Y$. Assume for contradiction that there exists a constant $M>0$ such that 
\[
c((\shz_n,\svt_n),\tvar_n) = \min \{ c((\shz,\svt_n),\tvar) : \shz \in \base, \, \tvar \in \Y \} \leq M \quad \forall \, n \in \N.
\]
Then, by continuity of $c$ (Assumption \ref{ass:cost}.\ref{ass:cost_cts}), for sufficiently large $n\in \N$, $c((\shz_*,\svt_n),\tvar_*)$ is bounded by $M+1$. This contradicts Assumption \ref{ass:cost}.\ref{ass:cost_inc} that $c((\shz_*,\cdot),\tvar_*)$ is unbounded. Hence,
\begin{equation}\label{eqn:c_to_infty}
    \lim_{n \to \infty} \min \{ c((\shz,\svt_n),\tvar) : \,\shz \in \base, \, \tvar \in \Y \}
= + \infty,
\end{equation}
and, in particular, there exists $P>0$ such that
\[
\min \{ c((\shz,P),\tvar) : \,\shz \in \base, \, \tvar \in \Y \}\geq \max_{\ev\in U}\ev.
\]
Then, by \eqref{eqn:q_max_min}, $\mathcal{Q}$ is uniformly bounded by $P$.

    \edit{It remains to show that $\mathcal{Q}$ is equicontinuous.} Let $P>0$ be a uniform upper bound for $\mathcal{Q}$. Define $\omega:[0,+\infty)\to[0,+\infty)$ by
    \[
    \omega(\delta) = \min\left\{ 
     c((\shz,\svt+\delta),\tvar) - c((\shz,\svt),\tvar)
    \,:\,
    (\shz,\svt,\tvar)\in \base\times [0,P]\times\Y
    \right\}.
    \]
    %This is well defined because $c$ is continuous and $\base\times [0,P]\times\Y$ is compact.
    We will use $\omega$ to construct a common modulus of continuity for functions in $\mathcal{Q}$.
    By Assumption \ref{ass:cost}.\ref{ass:cost_inc}, $c((\shz,\cdot),\tvar)$ is strictly increasing for all $(\shz,\tvar)\in \base\times \Y$, so $\omega$ is strictly increasing. By \eqref{eqn:c_to_infty} and the fact that $c$ is bounded on the compact set $\base\times [0,P]\times\Y$, $\omega$ is unbounded.
    Next, for arbitrary $n\in \N$, let $A=[0,n]$. 
    \edit{By Assumption \ref{ass:cost}.\ref{ass:cost_cts}, $c$ is Lipschitz on $(\base\times [0,P+n])\times \Y$, so the restriction of $\omega$ to $A$ is continuous; see e.g. \cite[Box 1.8]{santambrogio2015optimal}}. Since $n$ was arbitrary, this implies that $\omega$ is continuous. Since $\omega$ is strictly increasing, unbounded, and continuous, it is invertible and its inverse $\omega^{-1}:[0,+\infty)\to[0,+\infty)$ is strictly increasing and continuous. In particular, since $\omega(0)=0$, $\omega^{-1}$ is continuous at $0$ and $\underset{\delta\to 0}{\lim}\omega^{-1}(\delta)=0$.

    Let $\tvar\in\Y$ and $(\shz_i,\ev_i)\in \base\times U$ for $i\in\{1,2\}$ and denote $\bmalt(\shz_i,\tvar,\ev_i)$ by $\bmalt_i$. Then $\bmalt_i\leq P$ for $i\in \{1,2\}$. Since $c$ is locally Lipschitz, there exists a  Lipschitz constant $L>0$ for $c$ on $(\base\times [0,P])\times \Y$. Therefore, by construction,
    \begin{align}\label{eqn:q_diff_bound}
    \begin{split}
        \omega(\vert \bmalt_2-\bmalt_1 \vert)
        &\leq 
        \left\vert c((\shz_2,\bmalt_2),\tvar) - c((\shz_2,\bmalt_1),\tvar)\right\vert\\
        &\leq \left\vert c((\shz_2,\bmalt_2),\tvar) - c((\shz_1,\bmalt_1),\tvar)\right\vert +
        \left\vert c((\shz_1,\bmalt_1),\tvar) - c((\shz_2,\bmalt_1),\tvar)\right\vert\\
        &\leq \left\vert \ev_2 - \ev_1 \right\vert + L\|\shz_2 - \shz_1\|_{\R^{d-1}}
        \\&\leq (1+L)\| (\shz_2,\ev_2) - (\shz_1,\ev_1)\|_{\R^{d}}.
    \end{split}
    \end{align}
    Since $\omega$ is invertible and its inverse is strictly increasing, applying $\omega^{-1}$ to both sides of \eqref{eqn:q_diff_bound} gives
    \begin{equation}\label{eqn:mod_cont}
    \vert \bmalt(\shz_2,\tvar,\ev_2)-\bmalt(\shz_1,\tvar,\ev_1) \vert = 
    \vert \bmalt_2-\bmalt_1 \vert \leq \omega^{-1}\left((1+L)\| (\shz_2,\ev_2) - (\shz_1,\ev_1)\|_{\R^{d}}\right).
    \end{equation}
    Define the function $\tilde{\omega}:[0,+\infty)\to[0,+\infty)$ by $\tilde{\omega}(\delta)=\omega^{-1}((1+L)\delta)$. Then $\underset{\delta\to 0}{\lim}\tilde{\omega}(\delta)=0$, and by \eqref{eqn:mod_cont}, $\tilde{\omega}$ is a modulus of continuity for every function in $\mathcal{Q}$, which implies that $\mathcal{Q}$ is equicontinuous.
\end{proof}

\edit{We now define the candidate minimiser $\bm_\psi$ of the inner-dual functional $\sfct_\psi$.}

\begin{lem}\label{lem:bm}
\edit{For each $\psi\in\cts{\Y}$, there exists a unique constant $\cst_\psi\in\R$ such that the function
\begin{align}
\nonumber
&\bm_\psi:\base\to\R\\
\label{eqn:bm}
&\bm_\psi(\shz) = \max_{\tvar\in\Y}\{\bmalt(\shz,\tvar,\psi(\tvar)+\cst_\psi)\}
\end{align}
is a probability density, where $\bmalt$ is defined by \eqref{eqn:q}.}
\end{lem}

\begin{proof}
    Let $\psi\in \cts{\Y}$, and let $\cst_0,\,\cst_1\in \R$ satisfy $\cst_0\leq \cst_1$. 
    Since $\Y$ is compact and $\psi$ is continuous, there exists a compact set $U\subset\R$ such that $\psi(\tvar)+\cst\in U$ for all $(\tvar,\cst)\in \Y\times [\cst_0,\cst_1]$. By Lemma \ref{lem:Q_unifbdd_equicts}, the set
    \[
    \mathcal{Q}\coloneqq \{\bmalt(\cdot,\tvar,\cdot)\vert_{\base\times U}\}_{\tvar\in\Y}
    \]
    is uniformly bounded and equicontinuous. This implies that the functions defined for each $\tvar\in\Y$ by
    \begin{align*}
        &\bmalt_{\psi,\tvar}:\base\times [\cst_0,\cst_1]\to[0,+\infty),\\
        &\bmalt_{\psi,\tvar}(\shz,\cst) = \bmalt(\shz,\tvar,\psi(\tvar)+\cst),
    \end{align*}
    are uniformly bounded and equicontinuous.
    \edit{The function}
    \begin{align*}
        &\bmalt_{\psi}:\base\times [\cst_0,\cst_1]\to\R,\\
        &\bmalt_{\psi}(\shz,\cst) = \max_{\tvar\in\Y}\{\bmalt(\shz,\tvar,\psi(\tvar)+\cst)\},
    \end{align*}
    is therefore well defined, continuous, and bounded\edit{; see for example \cite[Box 1.8]{santambrogio2015optimal}}.
    It follows by the Dominated Convergence Theorem that the function $\mathcal{I}:\R\to\R$ defined by
    \[
    \mathcal{I}(\cst) \coloneqq \int_\base \bmalt_\psi(\shz,\cst)\, \rd \shz
    \]
    is continuous.
    
    We now find $\cst_0,\,\cst_1\in \R$ such that $\mathcal{I}(\tau_0)=0$ and $\mathcal{I}(\tau_1)\geq 1$, and we show that $\mathcal{I}$ is strictly increasing wherever it is positive.
    Then by the Intermediate Value Theorem, there exists a unique constant $\tau_\psi$ such that $\mathcal{I}(\tau_\psi)=1$, which,
    by non-negativity of $\bmalt$, is therefore the unique constant such that $\bm_\psi = \bmalt_\psi(\cdot,\cst_\psi)$ is a probability density.
    
    Let $r=1/\leb{d-1}(\base)$ and define constants
    \begin{equation}\label{eqn:helpful_constants}
    m\coloneqq \min_{(\shz,\tvar)\in\base\times\Y}\, c((\shz,0),\tvar),\qquad M\coloneqq \max_{(\shz,\tvar)\in\base\times\Y}\, c((\shz,r),\tvar),
    \end{equation}
    %which are well defined because $c$ is continuous and $\base\times\Y$ is compact
    \edit{and}
    \[
    \cst_0 \coloneqq m - \max_{\tvar\in\Y}\,\psi(\tvar),
    \qquad
    \cst_1 \coloneqq M - \max_{\tvar\in\Y}\, \psi(\tvar).
    \]
    %which depend on $\psi$, and are well defined since $\psi$ is continuous.
    By construction
    \[
    \psi(\tvar) + \cst_0 \leq c((\shz,0),\tvar) \qquad \forall \, \shz\in\base,\, \tvar\in\Y,
    \]
    so, by definition of $\bmalt$,
    \begin{align*}
    \bmalt_\psi(\shz,\cst_0) = \max_{\tvar\in\Y}\{\bmalt(\shz,\tvar,\psi(\tvar)+\cst_0)\} = 0 \qquad \forall\, \shz\in \base.
    \end{align*}
    Therefore, $\mathcal{I}(\cst_0)=0$. Since $c((\shz,\cdot),\tvar)$ is strictly increasing for every $\shz\in\base$ and $\tvar\in\Y$, and by definition of $M$ and $\cst_1$, for $\tvar\in\Y$ maximising $\psi$,
    \[
    M = \psi(\tvar) + \cst_1 \geq c((\shz,r),\tvar) \geq c((\shz,0),\tvar) \qquad \forall \, \shz\in\base.
    \]
    Since $c\big((\shz,\cdot),\tvar\big)^{-1}$ is increasing, this implies that
    \[\bmalt(\shz,\tvar,\psi(\tvar)+\cst_1) = c\big((\shz,\cdot),\tvar\big)^{-1}(\psi(\tvar) + \cst_1)\geq r \qquad \forall \, \shz\in\base.\]
    Then, by definition of $r$,
    \[
    \mathcal{I}(\cst_1) \geq \int_\base \bmalt(\shz,\tvar,\psi(\tvar)+\cst_1)\,\rd\shz \geq \int_\base r\,\rd\shz = 1.
    \]
    So $\mathcal{I}(\cst_0)=0 \leq 1\leq \mathcal{I}(\cst_1)$, as required.

    Let $\tau_\psi\in\R$ satisfy $\mathcal{I}(\tau_\psi) = 1$. We claim that this uniquely defines $\cst_\psi$. By the Intermediate Value Theorem, there exists $\tau\in\R$ such that $0<\mathcal{I}(\tau)<1$, so it is sufficient to show that $\mathcal{I}$ is strictly increasing on $[\cst,+\infty)$. By definition of $\mathcal{I}$ and non-negativity of $\bmalt$, there exists a set $A\subseteq \base$ with positive $\leb{d-1}$ measure such that
    \[
    0<\bmalt_{\psi}(\shz,\cst) \qquad \forall\, \shz\in A.
    \]
    Let $\shz\in A$ and take $\tvar\in\Y$ to be a maximiser of $\bmalt(\shz,\cdot,\psi(\tvar)+\cst)$ over $\Y$. 
    Then by definition of $\bmalt$, $c((\shz,0),\tvar) < \psi(\tvar) + \cst$ and
    \[
    \bmalt_{\psi}(\shz,\cst) = \bmalt(\shz,\tvar,\psi(\tvar)+\cst) = c((\shz,\cdot),\tvar)^{-1}(\psi(\tvar) + \cst).
    \]
    Since $c\big((\shz,\cdot),\tvar\big)^{-1}$ is strictly increasing, for any $\delta>0$
    \[
    \bmalt_{\psi}(\shz,\cst) = c((\shz,\cdot),\tvar)^{-1}(\psi(\tvar) + \cst) < c((\shz,\cdot),\tvar)^{-1}(\psi(\tvar) + \cst +\delta) \leq \bmalt_{\psi}(\shz,\cst+\delta).
    \]
    It follows that $\mathcal{I}$ is strictly increasing on $[\tau,+\infty)$, as required.
\end{proof}

\edit{We now prove that $\bm_\psi$ is the unique minimiser of $\sfct_{\psi}$.}

\begin{thm}
\label{thm:subdual}
\edit{For $\psi\in\cts{\Y}$, $\bm_\psi$ defined by \eqref{eqn:bm} is the unique minimiser of $\sfct_{\psi}$ over $\Pac(\base)$. In particular, $\bm=\bm_\psi$ satisfies \eqref{eqn:opt_cond_p} with $\cst=\cst_\psi$.}
\end{thm}
    
\begin{proof}
    \edit{By Lemma \ref{lem:bm}, $\bm_\psi$ is a well-defined probability density.
    By construction, $\bm=\bm_\psi$ satisfies \eqref{eqn:opt_cond_p} with $\cst=\cst_\psi$. We now verify that $\bm_\psi$ is the unique minimiser of 
    $\sfct_{\psi}$ over $\PM(\base)$.}
    \edit{Since $f_{\psi}(\shz,\cdot)$ is strictly convex on $[0,+\infty)$} and its partial derivative with respect to $\ul$ is $\psi^c(\shz,\cdot)$, for any distinct $\ula,\, \ul\in [0,+\infty)$,
    \begin{equation}\label{eqn:str_cvex}
    f_\psi(\shz,\ula)>f_\psi(\shz,\ul) + \psi^c(\shz,\ul)(\ula-\ul).
    \end{equation}
    \edit{Let $\bm_0\in\Pac(\base)$.} Suppressing the argument $\shz$, \eqref{eqn:opt_cond_p} and \eqref{eqn:str_cvex} give
    \begin{align*}
    \sfct_{\psi}(\bm_0)&=\int_{\base}f_\psi(\cdot,\bm_0)\, \rd\shz\\
    &>\int_{\base}\left[f_\psi(\cdot,\bm) + \psi^c(\cdot,\bm)(\bm_0-\bm)\right]\, \rd\shz\\
    &=\int_{\base}f_\psi(\cdot,\bm) \, \rd \shz 
    	+ \cst\int_{\base\cap \{\bm>0\}}(\bm_0-\bm)\, \rd\shz 
    	+ \int_{\base\cap \{\bm=0\}}\psi^c(\cdot,\bm)\bm_0\, \rd\shz\\
    &\geq\int_{\base}f_\psi(\cdot,\bm) \, \rd \shz 
    	+ \cst\int_{\base\cap \{\bm>0\}}(\bm_0-\bm)\, \rd\shz 
    	+ \int_{\base\cap \{\bm=0\}}\cst\bm_0\, \rd\shz\\
    &=\int_{\base}f_\psi(\cdot,\bm) \, \rd \shz 
    	+ \cst\int_{\base}(\bm_0-\bm)\, \rd\shz\\
    &=\int_{\base}f_\psi(\cdot,\bm) \, \rd \shz\\
    &=\sfct_{\psi}(\bm).
    \end{align*}
    Here, the final inequality holds because $\bm_0$ is non-negative, and the penultimate equality holds because $\bm$ and $\bm_0$ are probability densities.
\end{proof}

\edit{We conclude this subsection by proving that the set of all functions $\bm_\psi$ for $\psi\in\cts{\Y}$ is uniformly bounded and equicontinuous. This result will be used to restrict the dual functional to a subset of $\cts{\Y}$ that has compact closure, which allows us to prove the existence of maximisers of the dual functional; see Lemma \ref{lem:restriction} and Theorem \ref{thm:dual}. It will also be used in Section \ref{sect:stability} to prove stability of minimisers of Problem \ref{prob:primal}.}

\begin{lem}\label{lem:unif_bdd_equicts}
The set
\[
\mathcal{S} \coloneqq \{\bm_\psi\;\vert\; \psi\in \cts{\Y}\}
\]
is uniformly bounded and equicontinuous.
\end{lem}

\begin{proof}
    Let $m$ and $M$ be defined as in \eqref{eqn:helpful_constants}. We show in the proof of Lemma \ref{lem:bm} that
    \[
    \cst_\psi\leq  M - \max_{\tvar\in\Y}\, \psi(\tvar)
    \]
    This implies the uniform bound
    \begin{align*}
    \psi(\tvar) + \cst_\psi \leq 
    M \qquad \forall\, \tvar\in\Y\;\forall\,\psi\in \cts{\Y}.
    \end{align*}
    Moreover, by definition of $\bmalt$,
    \[
    \bmalt(\cdot,\tvar,\ev) = 0 \qquad \forall\, \tvar\in\Y\;\forall\, \ev\leq m.
    \]
    Therefore
    \[
    \mathcal{S}_0 \coloneqq \{q(\cdot,\tvar,\psi(\tvar)+\cst_\psi) \, \vert\, \tvar\in\Y,\, \psi\in\cts{\Y}\} \subseteq \{q(\cdot,\tvar,\ev) \, \vert\, \tvar\in\Y,\, \ev\in[m,M]\}\cup \{0\},
    \] 
    which is uniformly bounded and equicontinuous by Lemma \ref{lem:Q_unifbdd_equicts}. \edit{Since each $\bm_\psi$ is defined as the pointwise maximum over $y\in\Y$ of functions in $\mathcal{S}_0$}, the set $\mathcal{S}$ is also uniformly bounded and equicontinuous with the same bound and modulus of continuity; see e.g. \cite[Box 1.8]{santambrogio2015optimal}.
\end{proof}

\subsection{Maximisers of the dual functional}

We now show that maximisers of the dual functional $\df_\tm$ \eqref{eqn:dual} exist and satisfy a necessary and sufficient first-order optimality condition; see Theorem \ref{thm:dual}. \edit{For brevity, we drop the subcript $\tm$.} Recall that 
\[
\df(\psi) = \inf_{\bm\in\Pac(\base)}\, \kf(\,\cdot\,;\im_\bm,\tm),
\]
where $\kf(\,\cdot\,;\im_\bm,\tm)$ is the Kantorovich functional for source and target measures $\im_\bm$ and $\tm$, respectively.
Each Kantorovich functional is concave and upper semi-continuous. As such, $\df$ is concave and upper-semi continuous. The existence of a maximiser of $\df_\tm$ then follows by restriction to the set $A(\Y)$ defined in Lemma \ref{lem:restriction}, which we show to have compact closure. To establish the optimality condition, we apply \cite[Theorem 2.4.18]{zalinescu2002convex} to obtain an expression for the superdifferential of $\df_\tm$. We see that $\psi$ is a maximiser of $\df_\tm$ if and only if it is a Kantorovich potential from $\im_\bm$ to $\tm$ with $\bm=\bm_\psi$.

Let $P>0$ be a uniform bound for the set $\mathcal{S}=\{\bm_\psi\,\vert\, \psi\in \cts{\Y}\}$, which exists by Lemma \ref{lem:unif_bdd_equicts}. 
For $\varphi\in \cts{\base\times [0,P]}$, define $\varphi^{\bar{c}}:\Y\to\R$ by
\begin{align*}
\varphi^{\bar{c}}(\tvar)\coloneqq\min_{\svar\in \base\times [0,P]} \left\{c(\svar,\tvar)-\varphi(\svar)\right\}.
\end{align*}
%Here, the minimum is taken over the compact set $\base\times [0,P]$, so by continuity of $c$ and $\varphi$ it is attained.
We call $\varphi^{\bar{c}}$ the $\bar{c}$-transform of $\varphi$. (The notation $\bar{c}$ is simply used to indicate that the minimum is taken over the source space $\base\times [0,P]$, distinguishing the $\bar{c}$-transform from the the $c$-transform \eqref{eqn:c_transform}, which is defined by minimisation over the target space $\Y$.)
We say that a function $\psi\in \cts{\Y}$ is $\overline{c}$-concave if there exists $\varphi\in \cts{\base\times [0,P]}$ such that $\psi=\varphi^{\bar{c}}$.

%%%% Restriction to c-concave, min-zero potentials
\begin{lem}\label{lem:restriction}
Define
\begin{align*}
A(\Y) \coloneqq \{\psi\in \cts{\Y}\; \vert \; \psi \text{ is } \bar{c}\text{-concave, and } \min_{\tvar\in Y}\psi(\tvar) = 0\}.
\end{align*}
Then $A(\Y)$ is non-empty and
\begin{equation}\label{eqn:dual_equiv}
\sup_{\psi\in \cts{\Y}} \df(\psi) = \sup_{\psi\in A(\Y)} \df(\psi).
\end{equation}
\end{lem}

\begin{proof}
Let $\psi\in \cts{\Y}$. Since $B\times [0,P]$ is compact and $c$ is locally Lipschitz (Assumption \ref{ass:cost}.\ref{ass:cost_cts}), the functions $c(\cdot,\tvar)-\psi(\tvar)$ for $\tvar\in\Y$ are Lipschitz on $B\times [0,P]$ and all have the same Lipschitz constant. As the pointwise minimum of these functions, $\psi^c$ is Lipschitz on $\base\times [0,P]$. Likewise, $\psi^{c\overline{c}}:\Y\to\R$ defined by
\[
\psi^{c\overline{c}}(\tvar)\coloneqq \min_{\svar\in \base\times [0,P]} \left\{c(\svar,\tvar)-\psi^c(\svar)\right\}
\]
is also Lipschitz. In particular, letting $\lambda = \underset{\tvar\in\Y}{\min}\,\psi^{c\overline{c}}(\tvar)$,
\[\tilde{\psi}\coloneqq \psi^{c\overline{c}} - \lambda = \left(\psi + \lambda\right)^{c\overline{c}}\]
is continuous, $\overline{c}$-concave, and has minimum value zero, so is an element of $A(\Y)$. By \cite[Proposition 1.34]{santambrogio2015optimal}, $\psi^{c\overline{c}c}=\psi^c$ on $\base\times [0,P]$ and $\psi^{c\overline{c}}\geq \psi$. Then, for any $\bm\in \Pac(\base)$ with continuous density such that $\|\bm\|_{\cts{\base}}\leq P$, the support of $\im_\bm$ is contained in $\base\times[0,P]$ and
\begin{equation}\label{eqn:kant_equiv}
\begin{aligned}
\kf(\psi;\im_\bm,\tm) &= \int_\X\psi^c\,\rd\im_\bm + \int_\Y\psi\,\rd\tm \\
&\leq \int_\X \psi^{c\overline{c}c}\,\rd\im_\bm + \int_\Y\psi^{c\overline{c}}\,\rd\tm.\\
&= \int_\X \left(\psi^{c\overline{c}c} + \lambda\right)\,\rd\im_\bm + \int_\Y\left(\psi^{c\overline{c}} - \lambda\right)\,\rd\tm\\
&= \int_\X \tilde{\psi}^{c}\,\rd\im_\bm + \int_\Y\tilde{\psi}\,\rd\tm \\
&= \kf(\tilde{\psi};\im_\bm,\tm).
\end{aligned}
\end{equation}
Recall that by definition $\kf(\psi;\im_\bm,\tm) = \sfct_{\psi}(\bm)$, similarly for $\tilde{\psi}$. By Theorem \ref{thm:subdual}, each of the functionals $\sfct_{\psi}$ and $\sfct_{\tilde{\psi}}$ has a unique minimiser over $\Pac(\base)$, and by Lemma \ref{lem:unif_bdd_equicts} this minimiser is continuous and bounded by $P$. Taking the minimum over $\bm$ in \eqref{eqn:kant_equiv} therefore gives
\[
\df(\psi) = \min_{\bm\in\Pac(\base)}\kf(\psi;\im_\bm,\tm) \leq \min_{\bm\in\Pac(\base)}\kf(\tilde{\psi};\im_\bm,\tm) = \df(\tilde{\psi}).
\]
So, for any $\psi\in \cts{\Y}$ we have found $\tilde{\psi}\in A(\Y)$ such that $\df(\psi) \leq \df(\tilde{\psi})$, which establishes \eqref{eqn:dual_equiv}.
\end{proof}

In order to derive optimality conditions using \cite[Theorem 2.4.18]{zalinescu2002convex}, we now show that the map $\bm\mapsto\kf(\psi,\im_\bm,\tm)=H_{\psi}(\bm)$ is lower semi-continuous for every $\psi\in \cts{\Y}$ by applying the following special case of \cite[Theorem 3.3]{bouchitte1990new}.

%%%% general l.s.c. result
\begin{prop}[\normalfont{\cite[Theorem 3.3]{bouchitte1990new}}]\label{prop:f_infty_phi_f}
Let $f:\base\times \R \to [0,+\infty]$ be such that, for every $\shz\in\base$, $f(\shz,\cdot)$ is proper, convex, lower semi-continuous, and equal to $+\infty$ on $(-\infty,0)$. Suppose also that $\underset{t\to+\infty}{\lim}\frac{f(\cdot,t)}{t}=+\infty$
uniformly on $\base$.
Then the functional defined on $\PM(\base)$ by
\begin{align*}
\bm \mapsto \begin{cases}
\displaystyle{\int_{\base} }f(\shz,\bm(\shz))\, \rd \shz \quad & \text{if }\bm\in\Pac(\base)\\
+\infty \quad & \text{otherwise,}
\end{cases}
\end{align*}
is lower semi-continuous with respect to the weak-* topology on $\PM(\base)$.
\end{prop}

%%%% lower semi-continuity of H
\begin{lem}\label{lem:lsc}
For any $\tm\in\PM(\Y)$ and $\psi\in\cts{\Y}$ the functional $\sfct_{\psi}$ is lower semi-continuous with respect to the weak-* topology on $\PM(\base)$.
\end{lem}
\begin{proof}
It is sufficient to show that $f_\psi$ satisfies the hypotheses of Proposition \ref{prop:f_infty_phi_f}. \edit{For every $\shz\in\base$, $f_\psi(\shz,\cdot)$ is proper, convex, and lower semi-continuous. Since $\sfct_{\psi}(\bm) = \sfct_{\psi+\lambda}(\bm)$ for all $\lambda\in\R$ and $\Y$ is compact, we assume without loss of generality that $\psi(\tvar) < 0$ for all $\tvar\in\Y$. Then, since $c$ non-negative and $\psi$ is negative, $f_\psi$ is non-negative and equal to $+\infty$ on the interval $(-\infty,0)$.}

It remains to show that $\underset{t\to+\infty}{\lim}\frac{f_\psi(\cdot,t)}{t}=+\infty$ uniformly on $\base$. That is
\begin{align*}
\forall\, M>0,\; \exists\; T>0\; \text{ such that }\, t\geq T \implies \, \frac{f_\psi(x,t)}{t}\geq M\quad \forall \, \shz\in\base.
\end{align*}
Let $M>0$ and let $r\in\R$ satisfy
\[r> \max_{(\shz,\tvar)\in\base\times\Y} \bmalt(\shz,\tvar,M+\psi(\tvar))\geq 0. \]
This maximum is well-defined and non-negative because $\bmalt$ is  non-negative and continuous by \edit{Lemma} \ref{lem:Q_unifbdd_equicts} and $\base\times\Y$ is compact. By construction and Assumption \ref{ass:cost}.\ref{ass:cost_inc},
\[c((\shz,r),\tvar)>M+\psi(\tvar) \quad \forall\, \shz\in\base,\, \tvar\in\Y,\] 
which implies that
\[
-M + \underset{\shz\in\base}{\min}\,\psi^c(\shz,r) = -M + \underset{\shz\in\base}{\min}\min_{\tvar\in\Y}\{c((\shz,r),\tvar)-\psi(\tvar)\} >0.
\]
Define the \edit{positive} constant
\begin{align*}
T\coloneqq \frac{r\,\underset{\shz\in \base}{\max}\,\psi^c(\shz,r)}{-M + \underset{\shz\in \base}{\min}\,\psi^c(\shz,r)}.
\end{align*}
By \eqref{eqn:str_cvex}, for any $\shza\in\base$ and $t\geq T$,
\[
\frac{f_\psi(\shza,t)}{t} \geq \frac{f_\psi(\shza,r)+(t-r)\pdone{f_\psi}{t}(\shza,r)}{t} = \frac{f_\psi(\shza,r) - r \psi^c(\shza,r)}{t} + \psi^c(\shza,r).
\]
By definition of $f_\psi$ and positivity of $\psi^c$,
\[
f_\psi(\shza,r) - r \psi^c(\shza,r) = \int_{0}^r \left(\psi^c(\shza,\tilde{r}) - \psi^c(\shza,r)\right)\,\rd \tilde{r} > - r\max_{\shz\in\base}\psi^c(\shz,r)
\]
It follows that
\[
\frac{f_\psi(\shza,t)}{t}\geq \frac{ - r\, \underset{\shz\in\base}{\max}\,\psi^c(\shz,r)}{T} + \min_{\shz\in\base}\,\psi^c(\shz,t)=M,
\]
as required.
Applying Proposition \ref{prop:f_infty_phi_f} with $f=f_\psi$ completes the proof.
\end{proof}

We now state and prove the main result of this section.

\begin{thm}\label{thm:dual}
There exists a maximiser of $\df$. In particular, $\psi$ maximises $\df$ if and only if $\psi$ is a Kantorovich potential from $\im_{\bm}$ to $\tm$ with $\bm=\bm_\psi$.
\end{thm}

\begin{proof}
We first prove existence of a maximiser of $\df$. By Lemma \ref{lem:restriction}, it is sufficient to prove that the closure of $A(\Y)$ is a compact subset of $\cts{\Y}$, and that $\df$ is proper, concave, upper semi-continuous, and bounded above on $A(\Y)$.

Let $\psi\in A(\Y)$. Then there exists $\varphi\in \cts{\base\times [0,P]}$ such that $\psi = \varphi^{\overline{c}}$. By Assumption \ref{ass:cost}.\ref{ass:cost_cts}, the cost $c$ is Lipschitz on the compact set $(\base\times[0,P])\times \Y$ with some Lipschitz constant $L>0$. As the pointwise minimum over a family of $L$-Lipschitz functions, $\psi$ is $L$-Lipschitz on $\Y$. Then, since $\underset{\tvar\in\Y}{\min}{\,\psi(\tvar)} = 0$,
\begin{equation}\label{eqn:psi_bound}
    \vert\psi(\tvar)\vert \leq L\mathrm{diam}(\Y)\qquad \forall\,\tvar\in\Y,
\end{equation}
which is finite since $\Y$ is compact. As such, $A(\Y)$ is uniformly bounded and equicontinuous, so by the Arzela-Ascoli Theorem, its closure is compact in $\cts{\Y}$ with respect to the uniform topology.

Next we show that $\df$ is concave and upper semi-continuous with respect to the uniform topology. Let $\bm\in\Pac(\base)$ and let $\psi_0,\,\psi_1\in \cts{\Y}$. Let $\svar\in \X$, and let $\tvar_i\in\Y$ be such that
\[
\psi_i^c(\svar) = c(\svar,\tvar_i) - \psi_i(\tvar_i)
\]
for each $i\in \{0,1\}$.
Then, using the definition of the $c$-transform,
\begin{align*}
\psi^c_1(\svar) - \psi_0^c(\svar) &= \min_{\tvar\in\Y}\{c(\svar,\tvar)-\psi_1(\tvar)\} - \left(c(\svar,\tvar_0) - \psi_0(\tvar_0)\right)\\
&\leq \left(c(\svar,\tvar_0) - \psi_1(\tvar_0)\right)- \left(c(\svar,\tvar_0) - \psi_0(\tvar_0)\right)\\
&= \psi_1(\tvar_0) - \psi_0(\tvar_0)\\
&\leq \|\psi_1 - \psi_0\|_{\cts{\Y}}.
\end{align*}
By symmetry, it follows that
\[
\left\vert\psi_1^c(\svar) - \psi_0^c(\svar)\right\vert
= \max\left\{\psi_1^c(\svar) - \psi_0^c(\svar),\, \psi_0^c(\svar) - \psi_1^c(\svar)\right\}
\leq \|\psi_1 - \psi_0\|_{\cts{\Y}}.
\]
Since $\svar\in\X$ was arbitrary,
\[
\left\vert \kf(\psi_1;\im_\bm,\tm) - \kf(\psi_0;\im_\bm,\tm) \right\vert 
= \left\vert \int_\X \left(\psi_1^c - \psi_0^c\right) \, \rd \im_\bm + \int_\Y \left(\psi_1 - \psi_0\right) \, \rd \tm \right\vert 
\leq 2\left\|\psi_1 - \psi_0\right\|_{\cts{\Y}}.
\]
So $\kf(\,\cdot\,;\im_\bm,\tm)$ is $2$-Lipschitz on $\cts{\Y}$. In particular, it is upper semi-continuous. Moreover, for $\lambda\in(0,1)$ and $\svar\in\X$,
\begin{align*}
((1-\lambda)\psi_1+\lambda\psi_0)^c(\svar) &= \min_{\tvar\in Y}\left\{c(\svar,\tvar) - (1-\lambda)\psi_1(\tvar) - \lambda\psi_0(\tvar)\right\}\\
& \geq (1-\lambda)\min_{\tvar\in Y}\left\{c(\svar,\tvar) - \psi_1(\tvar)\right\} - \lambda\min_{\tvar\in Y}\left\{c(\svar,\tvar) - \psi_0(\tvar)\right\}\\
& = (1-\lambda)\psi_1^c(\svar) + \lambda\psi_0^c(\svar).
\end{align*}
Therefore,
\begin{align*}
\kf((1-\lambda)\psi_1+\lambda\psi_0;\im_\bm,\tm) 
&= \int_\X ((1-\lambda)\psi_1+\lambda\psi_0)^c \, \rd \im_\bm + \int_\Y (1-\lambda)\psi_1+\lambda\psi_0 \, \rd \nu\\
& \geq (1-\lambda)\left(\int_\X \psi_1^c \, \rd \im_\bm + \int_\Y \psi_1 \, \rd \nu\right) + \lambda\left( \int_\X\psi_0^c \, \rd \im_\bm + \int_\Y \psi_0 \, \rd \nu\right)\\
& = (1-\lambda)\kf(\psi_1;\im_\bm,\tm)+\lambda\kf(\psi_0;\im_\bm,\tm),
\end{align*}
so $\kf(\,\cdot\,;\im_\bm,\tm)$ is concave. As the pointwise minimum of a family of concave $2$-Lipschitz functions, $\df$ is concave and $2$-Lipschitz. In particular, it is upper semi-continuous.

We now show that $\df$ is real valued (in particular proper) and bounded above on $A(\Y)$. Indeed, let $\psi\in A(\Y)$ and let $\bm = \bm_\psi$. 
By Theorem \ref{thm:subdual}, $\bm_\psi$ minimises $\sfct_{\psi}$, so by definition of $\df$,
\begin{equation}\label{eqn:G_eq_H}
    \df(\psi) = \kf(\psi;\im_\bm,\tm).
\end{equation}
By Lemma \ref{lem:unif_bdd_equicts}, $\|\bm\|_{\cts{\base}}\leq P$,
where $P>0$ is a constant independent of $\psi$. 
Since $c$ is Lipschitz on the compact set $(\base\times [0,P])\times \Y$, its maximum over that set, which we denote by  $c_{\max}$, is attained. 
Using the bound \eqref{eqn:psi_bound} and the fact that $\underset{\tvar\in \Y}{\min}\, \psi(\tvar)=0$,
\begin{align*}
\kf(\psi;\im_\bm,\tm)
& = \int_\X \psi^c \, \rd \im_\bm + \int_\Y \psi \, \rd \nu\\
& \leq \left(c_{\max} - \min_{\tvar\in\Y} \psi(\tvar)\right)\int_\base\bm(\shz)\,\rd\shz + L\mathrm{diam}(\Y)\\
& = c_{\max} + L\mathrm{diam}(\Y),
\end{align*}
which is an upper bound for $\df$ on $A(\Y)$ by \eqref{eqn:G_eq_H}. In particular, since $\bm$ is bounded and $\psi$ and $\psi^c$ are continuous, $\df(\psi)$ is real valued so $\df$ is proper. Since this holds for any $\psi\in \cts{\Y}$, $\df$ is real valued. This concludes the proof that a maximiser of $\df$ over $\cts{\Y}$ exists.

Finally we prove the necessary and sufficient optimality condition. Since $\df$ is concave and real valued, for every $\psi\in \cts{\Y}$ the superdifferential of $\df$ at $\psi$, denoted by $\partial^+\df(\psi)$, is non-empty and contains $0$ if and only if $\psi$ is a maximiser of $\df$. By Lemma \ref{lem:lsc}, for every $\psi\in \cts{\Y}$ the functional $\sfct_{\psi}$ is lower semi-continuous with respect to the weak-* topology on $\PM(\base)$. By definition
\[
\df(\psi) = \inf_{\bm\in\PM(\Y)}\, \sfct_{\psi}(\bm).
\]
Let $\psi\in \cts{\Y}$ and let $\bm = \bm_\psi$ be the unique minimiser of $\sfct_{\psi}$, which exists by Theorem \ref{thm:subdual}. Then by \cite[Theorem 2.4.18]{zalinescu2002convex},
\[
\partial^+\df(\psi) = \partial^+\left[\kf(\,\cdot\,;\im_\bm,\tm)\right](\psi).
\]
Hence, $\psi$ is a maximiser of $\df$ if and only if $0\in\partial^+\left[\kf(\,\cdot\,;\im_\bm,\tm)\right](\psi)$, which holds if and only if $\psi$ is a Kantorovich potential from $\im_\bm$ to $\nu$.
\end{proof}

\subsection{Minimisers of the primal functional}

We now combine the results of this section to prove that the primal problem (Problem \ref{prob:primal}) has a unique solution, to prove that duality holds between the primal and dual problems, and to derive an optimality condition relating both problems.

\begin{thm}\label{thm:duality_existence_uniqueness}
There exists a unique minimiser of $\pf$. In particular,
\begin{align}\label{eqn:duality}
\min_{\bm\in\Pac(\base)}\pf(\bm)
=
\max_{\psi\in \cts{\Y}}\df(\psi),
\end{align}
and $\bm$ and $\psi$ are optimal in their respective problems if and only if $p=p_\psi$ and $\psi$ is a Kantorovich potential from $\im_{\bm}$ to $\tm$.
\end{thm}

\begin{proof}
First we prove duality and existence of a minimiser of $\pf$ \eqref{eqn:duality}. Let $\psi\in \cts{\Y}$ be a maximiser of $\df$, which exists by Theorem \ref{thm:dual}. By Theorem \ref{thm:subdual}, $\bm=\bm_\psi$ is the unique minimiser of $\sfct_{\psi}$, and the optimality condition stated in Theorem \ref{thm:dual} implies that $\psi$ is a Kantorovich potential from $\im_\bm$ to $\tm$. Hence, by the Kantorovich Duality Theorem,
\begin{align*}
\df(\psi) = \inf_{\Pac(\base)} H_{\tm,\psi} = \kf(\psi;\im_\bm,\tm) = \max_{\cts{\Y}}\kf(\,\cdot\,;\im_\bm,\tm) = \T(\im_\bm,\tm) = \pf(\bm).
\end{align*}
Then, using weak duality \eqref{eqn:weak_duality},
\begin{align}\label{eqn:duality_pf}
\max_{\cts{\Y}} \,\df = \df(\psi) = \pf(\bm) \geq \inf_{\Pac(\base)} \pf \geq \sup_{\cts{\Y}} \,\df,
\end{align}
So all inequalities in \eqref{eqn:duality_pf} must be equalities, meaning that $\bm$ is a minimiser of the primal functional $\pf$ over $\Pac(\base)$. This establishes the existence of a minimiser of $\pf$ over $\Pac(\base)$, and proves that \eqref{eqn:duality} holds.

%%%% uniqueness of minimisers of primal
By Theorem \ref{thm:subdual}, $\underset{\Pac(\base)}{\argmin}\sfct_{\psi} = \{\bm_\psi\}$ is a singleton for every $\psi\in \cts{\Y}$. To prove that the minimiser of $\pf$ is unique, it is therefore sufficient to prove that
\begin{align}\label{eqn:set_eq}
{\argmin_{\Pac(\base)}} \, \pf
= {\argmin_{\Pac(\base)}}\, \sfct_{\psi} \qquad \forall\, \psi\in \underset{\cts{\Y}}{\argmax}\, \df.
\end{align}
Let $\psi$ be a maximiser of $\df$ and let $\bm=\bm_\psi$ be the unique minimiser of $\sfct_{\psi}$. By Theorem \ref{thm:dual}, $\psi$ is a Kantorovich potential from $\im_\bm$ to $\tm$. Then by \eqref{eqn:duality} and the Kantorovich Duality Theorem,
\begin{align*}
{\min_{\Pac(\base)}}\, \pf 
= {\max_{\cts{\Y}}}\, \df
= \df(\psi)
= \min_{\Pac(\base)}\, \sfct_{\psi}
=\kf(\psi;\im_\bm,\tm)
= \T(\im_\bm,\tm)
= \pf(\bm).
\end{align*}
Hence
\begin{equation}\label{eqn:subset}
\underset{\Pac(\base)}{\argmin}\, \sfct_{\psi} \subseteq \underset{\Pac(\base)}{\argmin}\, \pf.
\end{equation}
For a contradiction, suppose that there exists 
$\bm\in \underset{\Pac(\base)}{\argmin}\, \pf \setminus \underset{\Pac(\base)}{\argmin}\, \sfct_{\psi}$. By the Kantorovich Duality Theorem,
\begin{align}\label{eqn:contr_1}
\min_{\Pac(\base)}\, \pf 
= \pf(\bm)
= \T(\im_\bm,\tm)
\geq \kf(\psi;\im_\bm,\tm)
= \sfct_{\psi}(\bm).
\end{align}
By assumption,
\begin{align}\label{eqn:contr_2}
\sfct_{\psi}(\bm)
> \underset{\Pac(\base)}{\min}\sfct_{\psi}
=\df_{\tm}(\psi)
= \underset{\cts{\Y}}{\max}\, \df.
\end{align}
Together, \eqref{eqn:contr_1} and \eqref{eqn:contr_2} imply that
\[
\underset{\Pac(\base)}{\min}\, \pf 
>\underset{\cts{\Y}}{\max}\, \df,
\]
which contradicts \eqref{eqn:duality}. Hence, it must hold that
\begin{equation}\label{eqn:superset}
\underset{\Pac(\base)}{\argmin}\, \sfct_{\psi} \supseteq \underset{\Pac(\base)}{\argmin}\, \pf.
\end{equation}
Together, \eqref{eqn:subset} and \eqref{eqn:superset} establish \eqref{eqn:set_eq}, as required.

%%%%% optimality conditions
We conclude the proof by establishing the necessary and sufficient optimality conditions. By \eqref{eqn:set_eq}, $\bm$ and $\psi$ are optimal in their respective problems if and only if
\[
\bm = \bm_\psi \qquad \text{and} \qquad \psi\in \underset{\cts{\Y}}{\argmax}\, \df.
\]
By Theorem \ref{thm:dual}, this holds if and only if $\psi$ is a Kantorovich potential from $\im_\bm$ to $\tm$, as required.
\end{proof}

\section{Stability}\label{sect:stability}

In this section, we establish stability of Problem \ref{prob:primal} with respect to the target measure $\tm$. We work in the setting described in Section \ref{sect:prob} and cost functions $c:\X\times\Y\to [0,+\infty)$ satisfying only Assumption \ref{ass:cost}, unless otherwise specified. Using the stability of the optimal transport cost $\T$ with respect to source and target measures \cite[Theorem 5.20]{villani2008optimal}, we show that for $\tm_\n$ converging to $\tm$ in $\PM(\Y)$, the sequence of minimisers $\bm_\n$ of the functionals $\pf_{\tm_\n}$ converges uniformly to the minimiser $\bm$ of $\pf_\tm$ (Theorem \ref{thm:min_conv}). If the cost $c$ additionally satisfies the twist condition (Assumption \ref{ass:twist}) then there exists a corresponding sequence of optimal transport maps $T_n$ between $\im_{\bm_\n}$ and $\tm_\n$, and we show that this sequence converges to the optimal transport map $T$ between $\im_{\bm}$ and $\tm$ (Theorem \ref{thm:transport_maps}).

%%%% Gamma convergence
\begin{thm}\label{thm:min_conv}
The map $\pf:\left(\Pac(\base)\cap L^{\infty}(\base),\|\cdot\|_{L^\infty(\base)}\right)\times \PM(\Y)\to\R$ given by
\[
\pf(\bm,\tm) \coloneqq \pf_\tm(\bm) = \T(\im_\bm,\tm)
\]
is continuous, where $\PM(\Y)$ is equipped with the weak topology. In particular, suppose that $\tm_\n\rightharpoonup\tm$ in $\PM(\Y)$, and let $\bm_\n$ (respectively $\bm$) be the unique minimiser of $\pf_{\tm_\n}$ (respectively $\pf_\tm$) over $\Pac(\base)$ for each $\n\in\N$. Then $\pf_{\tm_n} \overset{\Gamma}{\to} \pf_\tm$ on $\left(\Pac(\base)\cap L^{\infty}(\base),\|\cdot\|_{L^\infty(\base)}\right)$, $\bm_n \to\bm$ in $\cts{\base}$, and \edit{$\im_{\bm_n} \to \im_{\bm}$} strongly in $L^1(\X)$.
Moreover, if $\gamma_n$ (respectively $\gamma$) is an optimal transport plan from $\im_{\bm_n}$ to $\tm_n$ (respectively from $\im_{\bm}$ to $\tm$) for each $n\in\N$, then up to a subsequence $\gamma_n\rightharpoonup \gamma$ in $\PM(\X\times\Y)$.
\end{thm}

\begin{proof}
We first show that $F$ is continuous. Let $\left((\bm_\n,\tm_\n)\right)_{n\in\N}$ be a sequence converging in \newline
$\left(\Pac(\base)\cap L^{\infty}(\base),\|\cdot\|_{L^\infty(\base)}\right)\times \PM(\Y)$ to a point $(\bm,\tm)$. Since $\base$ is compact, $\|\bm_\n-\bm\|_{L^1(\base)}\to 0$, and by Remark \ref{rem:well_def}, $\|\im_{\bm_n}-\im_{\bm}\|_{L^1(\X)}\to 0 $. This implies that
\begin{equation}\label{eqn:im_weak_conv}
    \im_{\bm_n}\rightharpoonup \im_\bm.
\end{equation}
Since the sequence $(\bm_\n)_{\n\in\N}$ converges in $L^\infty(\base)$, it is uniformly bounded, so $\im_\bm$ and $\im_{\bm_\n}$ have common compact support, say $A\subset\X$. By continuity of $c$ (Assumption \ref{ass:cost}.\ref{ass:cost_cts}) and compactness of $A\times\Y$,
\begin{equation}\label{eqn:finite_cost}
\T(\im_{\bm},\tm)<+\infty\quad\text{and}\quad\T(\im_{\bm_\n},\tm_\n)<+\infty\quad \forall\,\n\in\N.
\end{equation}
For each $\n\in\N$ let $\gamma_\n$ be an optimal transport plan from $\im_{\bm_\n}$ to $\tm_\n$. Since $\im_{\bm_n}\rightharpoonup \im_\bm$ and $\tm_\n\rightharpoonup \tm$ and the corresponding transport costs are finite \eqref{eqn:finite_cost}, by \cite[Theorem 5.20]{villani2008optimal}, there exists an optimal transport plan $\gamma$ from $\im_\bm$ to $\tm$ such that, up to a subsequence, $\gamma_n\rightharpoonup\gamma$. For all $\n\in\N$, the support of $\gamma_n$ and $\gamma$ is necessarily contained in $A\times\Y$, on which $c$ is a bounded continuous function, so, up to a subsequence,
\[
\pf(\bm_\n,\tm_\n) = \T(\im_{\bm_\n},\nu_\n) = \int_{A\times \Y} c \, \rd\gamma_\n  \to \int_{A\times \Y} c \, \rd\gamma =  \T(\im_{\bm},\nu) =\pf(\bm,\tm).
\]
Since the limit is independent of the chosen subsequence, the whole sequence converges to this limit. This establishes the claimed continuity of $F$, which immediately implies the claimed $\Gamma$-convergence.

Finally, suppose that $\tm_\n\to\tm$ in $\PM(\Y)$, let $\bm_\n$ be the unique minimiser of $\pf_{\tm_\n}$ over $\Pac(\base)$ for each $\n\in\N$, and let $\bm$ be the unique minimiser of $\pf_\tm$ over $\Pac(\base)$. By Theorem \ref{thm:duality_existence_uniqueness}, $(\bm_\n)_{\n\in\N}\subset \mathcal{S}$, which is uniformly bounded and equicontinuous by Lemma \ref{lem:unif_bdd_equicts}. By the Ascoli-Arzela Theorem, any subsequence has a further subsequence that converges uniformly, and therefore with respect to the strong $L^\infty$ topology. Since $\pf_{\tm_\n}\overset{\Gamma}{\to}\pf_\tm$ on $(\Pac(\base)\cap L^\infty(\base),\|\cdot\|_{L^\infty(\base)})$, the limit of this further subsequence must be equal to $\bm$. Since the limit is independent of the choice of subsequence, the whole sequence converges to this limit. By Remark \ref{rem:well_def}, $\im_{\bm_n}\to\im_\bm$ strongly in $L^1(\X)$. Letting $\gamma_n$ (respectively $\gamma$) by an optimal transport plan from $\im_{\bm_n}$ to $\tm_n$ (respectively from $\im_\bm$ to $\tm$) for each $n\in\N$, applying \cite[Theorem 5.20]{villani2008optimal} as above gives that, up to a subsequence, $\gamma_n\rightharpoonup\gamma$ in $\PM(\X\times\Y)$.
\end{proof}

Theorem \ref{thm:min_conv} establishes stability of minimisers of Problem \ref{prob:primal} with respect to $\tm$. We now show that when the cost $c$ is twisted (Assumption \ref{ass:twist}), if $\bm$ is the minimiser of $\pf_\tm$ then there is a unique optimal transport map $T$ from $\im_\bm$ to $\tm$, and $T$ is also stable with respect to $\tm$.

\begin{thm}[Existence and stability of transport maps]\label{thm:transport_maps}
Suppose that the cost $c:\X\times\Y\to\R$ satisfies Assumption \ref{ass:twist}. Then for any $\bm\in\Pac(\base)\cap L^{\infty}(\base)$ and $\tm\in\PM(\Y)$ there exists a unique optimal transport map $T$ from $\im_\bm$ to $\tm$, which is defined $\im_\bm$ almost everywhere. Moreover, suppose that $\tm_\n\rightharpoonup\tm$ in $\PM(\Y)$, let $(\bm_\n)_{\n\in\N}$ be the corresponding sequence of minimisers of Problem \ref{prob:primal}, and for each $n\in\N$ let $T_\n$ (respectively $T$) be the unique optimal transport map from $\im_{\bm_\n}$ to $\tm_\n$ (respectively $\im_\bm$ to $\tm$) that is zero outside the support of  $\im_{\bm_\n}$ (respectively $\im_\bm$). Then
\[
\|T_\n - T \|_{L^r(\X)} \to 0 \qquad \forall \, r\in [1,+\infty).
\]
\end{thm}

\begin{proof}
For any $\bm\in\Pac(\base)\cap L^{\infty}(\base)$ and $\tm\in\PM(\Y)$, both $\im_\bm$ and $\tm$ have compact support, and $\im_\bm$ is absolutely continuous with respect to $\leb{d}$. Since $c$ is twisted (Assumption \ref{ass:twist}), by the Gangbo-McCann Theorem (see for example \cite{gangbo1996geometry} or \cite[Theorem 12]{merigot2021optimal}), there exists a unique optimal transport map $T$ from $\im_\bm$ to $\tm$, which is defined $\im_\bm$ almost everywhere.
Moreover, the measure $\gamma\coloneqq (\mathrm{id}_{\X},T)_\#(\im_{\bm})$ is the unique optimal transport plan from $\im_{\bm}$ to $\tm$.

Suppose that $\tm_\n\rightharpoonup\tm$ in $\PM(\Y)$, let $(\bm_\n)_{\n\in\N}$ be the corresponding sequence of minimisers of Problem \ref{prob:primal}, and for each $n\in\N$ let $T_\n$ (respectively $T$) be the unique optimal transport map from $\im_{\bm_\n}$ to $\tm_\n$ (respectively $\im_\bm$ to $\tm$) extended by zero onto $\X$. Define $\gamma_\n\coloneqq (\mathrm{id}_{\X},T_\n)_\#(\im_{\bm_\n})$ and $\gamma\coloneqq (\mathrm{id}_{\X},T)_\#(\im_{\bm})$. Let $r\in [1,+\infty)$ and $\eps\in (0,1)$. We aim to find $N\in\N$ such that for $\n\geq N$, $\|T-T_n\|^r_{L^r(\X)}\leq \eps$. 

By Theorem \ref{thm:duality_existence_uniqueness}, $\bm_n,\,\bm\in \mathcal{S}$ for all $n\in\N$, and by Lemma \ref{lem:unif_bdd_equicts}, $\mathcal{S}$ is equicontinuous and uniformly bounded. Let $P>0$ be a uniform bound for $\mathcal{S}$. Define the compact set 
\[\X_P\coloneqq \{(\shz,\svt)\in X \, : \, \svt\leq P\},\]
and for each $\n\in \N$ define
\[
\Delta_{\n,\eps} \coloneqq \left\{\svar\in\X_P\,:\, \|T(\svar) - T_n(\svar)\|\geq \frac{\eps}{3}\right\}.
\]
We first show that with
\[
\delta = \frac{\eps}{3\mathrm{diam}(\Y)^r}
\]
there exists $N_0\in \N$ such that for $\n\geq N_0$, $\mu_\bm(\Delta_{\n,\eps})\leq\delta.$

By Theorem \ref{thm:min_conv},
$\|\im_\bm - \im_{\bm_\n}\|_{L^1(\X)}\to 0$ and $\gamma_\n\rightharpoonup\gamma$, where the convergence is along the whole sequence by uniqueness of $\gamma$. Hence, there exists $N\in\N$ such that for all $\n\geq N$
\begin{equation}\label{eqn:delta_1_4}
\|\im_\bm - \im_{\bm_\n}\|_{L^1(\X)}\leq \frac{\delta}{4}.
\end{equation}
By Lusin's Thoerem, there exists a compact set $A\subset \X_P$ such that $T\vert_{A}$ is continuous and
\begin{equation}\label{eqn:delta_2_4}
\im_\bm(\X\setminus A)\leq \frac{\delta}{4}.
\end{equation}
In particular, the set
\[
A_\eps \coloneqq \left\{(\svar,\tvar)\in A\times Y\, : \, \|T(\svar)-\tvar\|\geq\frac{\eps}{3}\right\}
\]
is closed. Since $\gamma$ is concentrated on the graph of $T$, $\gamma(A_\eps) = 0$, so weak convergence of $\gamma_\n$ to $\gamma$  and closedness of $A_\eps$ implies that there exists $N_1\in \N$ such that for $\n\geq N_1$
\begin{equation}\label{eqn:delta_3_4}
\gamma_\n(A_\eps) \leq \frac{\delta}{4}.
\end{equation}
We have,
\begin{align*}
\im_{\bm_\n}(\Delta_{\n,\eps}\cap A)
&= \im_{\bm_\n}\left(\left\{\svar\in A\,:\, \|T(\svar) - T_n(\svar)\|\geq \frac{\eps}{3}\right\}\right)\\
&=\gamma_\n\left(\left\{\svar\in A\,:\, \|T(\svar) - T_n(\svar)\|\geq \frac{\eps}{3}\right\}\times \Y\right)\\
&=\gamma_\n\left(\left\{(\svar,\tvar)\in A\times\Y\,:\, \|T(\svar) - \tvar\|\geq \frac{\eps}{3}\right\}\right)\\
&=\gamma_\n(A_\eps),
\end{align*}
where the second equality holds because $\gamma_\n$ has first marginal $\im_{\bm_\n}$, and the third equality holds because $\spt(\gamma_n)\subseteq\mathrm{graph}(T_\n)$.
For all $\n\geq N_2 \coloneqq \max\{N_0,N_1\}$,
\begin{align*}
\gamma_\n(A_\eps) = \im_{\bm_\n}(\Delta_{\n,\eps}\cap A) 
& = \im_{\bm_\n}(\Delta_{\n,\eps} \setminus (\X_P \setminus A)) \\
& \geq \im_{\bm_\n}(\Delta_{\n,\eps}) - \im_{\bm_\n}(\X_P \setminus A) \\
& = \im_{\bm_\n}(\Delta_{\n,\eps}) + \left(\im_{\bm}(\X_P \setminus A) - \im_{\bm_\n}(\X_P \setminus A)\right) - \im_{\bm}(\X_P \setminus A) \\
& \geq \im_{\bm_\n}(\Delta_{\n,\eps}) - \|\im_{\bm} - \im_{\bm_\n}\|_{L^1(\X)} - \im_{\bm}(\X_P\setminus A).
\end{align*}
Then by \eqref{eqn:delta_1_4}, \eqref{eqn:delta_2_4}, and \eqref{eqn:delta_3_4},
\begin{equation}\label{eqn:mu_n_Delta_bound}
    \im_{\bm_\n}(\Delta_{\n,\eps}) \leq \frac{3\delta}{4}.
\end{equation}
Then by \eqref{eqn:delta_1_4} and \eqref{eqn:mu_n_Delta_bound}, for all $n\geq N_2$,
\[
\im_{\bm}(\Delta_{\n,\eps}) = \im_{\bm}(\Delta_{\n,\eps}) -  \im_{\bm_\n}(\Delta_{\n,\eps}) +  \im_{\bm_\n}(\Delta_{\n,\eps}) \leq \|\im_{\bm} - \im_{\bm_\n}\|_{L^1(\X)} + \im_{\bm_\n}(\Delta_{\n,\eps}) \leq \delta,
\]
as claimed.

Since the optimal transport maps $T_\n$ and $T$ map into $\Y$
\[
\|T-T_\n\|_{L^\infty(\X)} \leq \mathrm{diam}(\Y).
\]
By definition of $\Delta_{\n,\eps}$, and since $\im_\bm$ is a probability measure on $\X$, it follows that for all $n\geq N_2$,
\begin{align*}
\|T - T_n \|^r_{L^r(\X,\im_{\bm})} &= \int_\X \| T- T_\n \|_{\R^d}^r \,\ \rd \im_\bm\\
& = \int_{\X\setminus \Delta_{\n,\eps}} \| T- T_\n \|_{\R^d}^r \,\rd \im_\bm + \int_{\Delta_{\n,\eps}} \| T- T_\n \|_{\R^d}^r \, \rd \im_\bm\\
& \leq \left(\frac{\eps}{3}\right)^r + \im_\bm(\Delta_{\n,\eps}) \mathrm{diam}(\Y)^r\\
& < \frac{\eps}{3} + \delta \mathrm{diam}(\Y)^r\\
& = \frac{2\eps}{3}.
\end{align*}
Finally, since $\Y$ is compact, there exists $R>0$ such that $\Y\subset B_R(0)\subset \R^d$. Uniform convergence of $\bm_{\n}$ to $\bm$ implies that there exists $N_3\in\N$ such that for all $\n\geq N_3$
\[
\| \bm - \bm _\n\|_{L^1(\base)} \leq \frac{\eps}{3R^r}.
\]
Let $N = \max\{N_2,N_3\}$. Since $T$ is zero outside of $\spt(\im_\bm)$, and $T_\n$ is zero outside of $\spt(\im_{\bm_\n})$, for $n\geq N$
\begin{align*}
\|T - T_\n \|^r_{L^r(\X)} &= \int_{\X} \|T(\svar) - T_\n(\svar) \|_{\R^d}^r \, \rd \svar\\
& = \|T - T_n \|^r_{L^r(\X,\im_\bm)} + \int_{\X} \|T_\n(\svar) \|_{\R^d}^r \chfun_{\{\bm(\shz)\leq\svt\leq \bm_\n(\shz)\}}(\svar)\, \rd \svar\\
& \leq \frac{2\eps}{3} + R^r \| \bm - \bm _\n\|_{L^1(\base)}\\
&= \eps,
\end{align*}
which completes the proof.
\end{proof}

\section{Reduction to an optimal transport problem}\label{sect:reduct}

\begin{figure}[t]
    \centering
    \begin{subfigure}[b]{0.44\textwidth}
        \includegraphics[width=\linewidth]{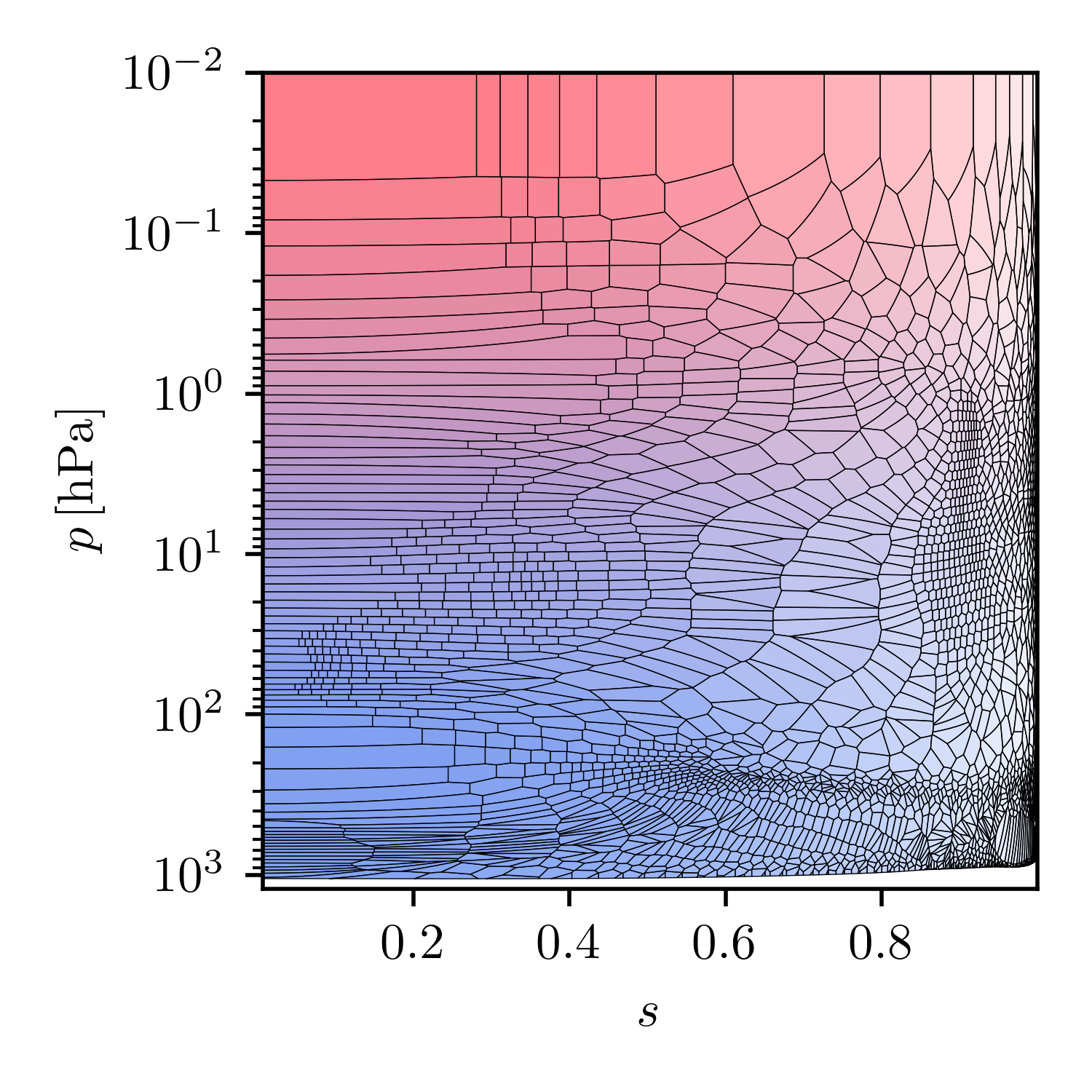}
        \caption{Optimal $\ext{c}$-Laguerre tessellation of $\X$}
    \end{subfigure}
    \hspace{.1\textwidth}
    \begin{subfigure}[b]{0.44\textwidth}
        \includegraphics[width=\linewidth]{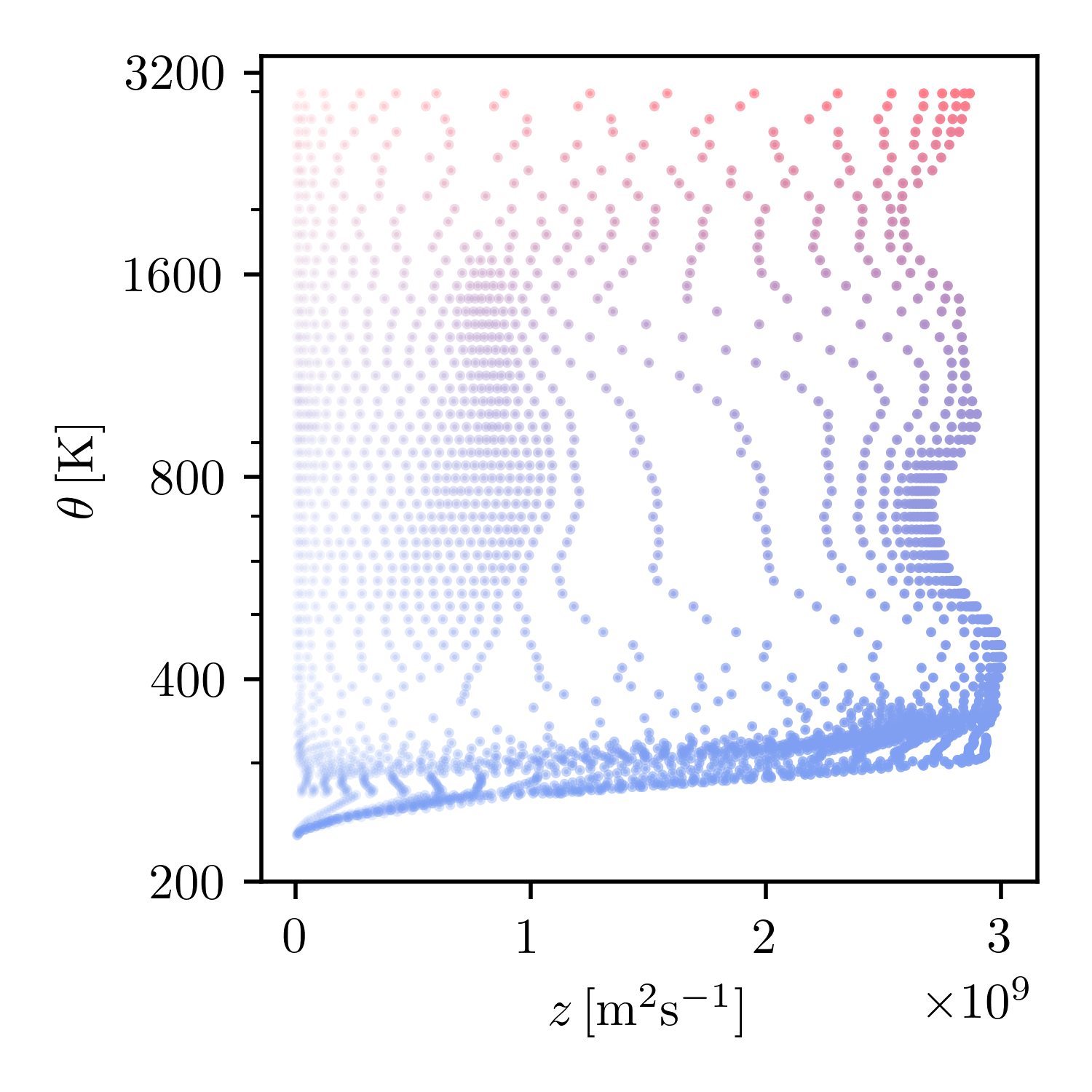}
        \caption{Discrete mass density $\nu$ on $\Y$}
    \end{subfigure}
    \caption{\edit{The energy-minimising MLM state computed from ERA-Interim data \cite{dee2011era} at 1800UTC on 22 January 2010. Right: the discrete mass density $\nu$ defined by \eqref{eqn:nu_computed} using the 3-D mass and circulation integrals \eqref{eqn:mass_int}-\eqref{eqn:circ_int} as computed in \cite{methven2015slowly}. Left: the optimal $\ext{c}$-Laguerre tessellation computed using the damped Newton method of \cite{kitagawa2019convergence}. The pressure axis is reversed so that the fictitious cell $L_{n+1}$, which represents the ground, extends across the bottom of the domain. Its top boundary coincides with the graph of the energy-minimising surface pressure $\bm$. Here, $p_{\mathrm{min}}=0.01\, \mathrm{hPa}$. The colour scheme is defined so that potential temperature increases from blue to red, and zonal angular momentum increases with opacity. We see that $\ZAM$ decreases poleward and $\PT$ decreases with pressure.}}
    \label{fig:computed_tessellation}
\end{figure}

\edit{We now use the optimality conditions of Theorem \ref{thm:duality_existence_uniqueness} to show that under Assumption \ref{ass:cost} Problem \ref{prob:primal} reduces to a standard optimal transport problem. This reduction is illustrated in Figure \ref{fig:reduction_schematic}.
In addition, we show that if the cost function $c$ is twisted (Assumption \ref{ass:twist}) and its gradient with respect to $\svar$ is bounded away from zero (Assumption \ref{ass:ext_twist}), then the corresponding Monge problem admits a unique solution (Corollary \ref{cor:reduction_map}). In the case where $\tm$ is discrete, we derive the form of this solution in terms of \emph{Laguerre cells} (Corollary \ref{cor:reduction_semi_discrete}), which can be found by solving a semi-discrete optimal transport problem \cite[Section 4]{merigot2021optimal}. This result gives rise a numerical method for computing energy-minimising MLM states from data. An example numerical solution using the background state cost function \eqref{eqn:bgs_cost} is illustrated in Figure \ref{fig:computed_tessellation}. The discrete target measure $\nu$ is defined by \eqref{eqn:nu_computed} using the 3-D mass and circulation integrals \eqref{eqn:mass_int} and \eqref{eqn:circ_int} computed from ERA-Interim data \cite{dee2011era} at 1800UTC on 22 January 2010, as in \cite{methven2015slowly}, making Figure \ref{fig:computed_tessellation} comparable to \cite[Figure 4]{methven2015slowly}. We solve the semi-discrete optimal transport problem numerically using the damped Newton method of \cite{kitagawa2019convergence}.}

\begin{thm}\label{thm:reduction}
Let $\ext{\tvar}\in\R^d\setminus\Y$ and define $\ext{c}:\X\times\Y\cup\{\ext{y}\}\to [0,+\infty)$ by
\[
\ext{c}(\svar,\tvar) = 
\begin{cases}
c(\svar,\tvar) \quad &\text{if}\quad \tvar\in \Y\\
0\quad & \text{if}\quad \tvar=\ext{\tvar}.
\end{cases}
\]
Let $P>0$ be a uniform upper bound on the set $\mathcal{S}$, and define
\[
\ext{\im}\coloneqq \leb{d}\mres(\base\times [0,P]),\qquad \ext{\tm}\coloneqq \tm + (P-1)\delta_{\ext{\tvar}}.
\]
For $\tm\in\PM(\Y)$,
\begin{align}\label{eqn:reduction}
\begin{rcases}
\bm \in\argmin \pf_\tm\\
\psi\in\argmax \df_\tm\\
\gamma \in \underset{\tilde{\gamma}\in\Gamma(\im_\bm,\tm)}{\argmin}\displaystyle{\underset{\X\times\Y}{\int}}c\, \rd\tilde{\gamma}
\end{rcases}
\iff
\begin{cases}
\bm = \underset{\tvar\in\Y}{\max}\, \bmalt(\cdot,\tvar,\ext{\psi}(\tvar)-\ext{\psi}(\ext{\tvar}))\\
\ext{\psi} \text{ is a Kantorovich potential from } \ext{\im} \text{ to }\ext{\tm}\\
\ext{\gamma} \in \underset{\tilde{\gamma}\in\Gamma(\ext{\im},\ext{\tm})}{\argmin}\displaystyle{\underset{\X\times\Y}{\int}}\ext{c}\, \rd\tilde{\gamma}
\end{cases}
\end{align}
where
\begin{equation}\label{eqn:gamma_ext}
\ext{\gamma}\coloneqq \gamma + \left(\leb{d}\mres \{(\shz,\svt)\, : \, \shz\in\base,\, \bm(\shz) \leq \svt\leq P \}\right)\otimes \delta_{\ext{\tvar}},
\end{equation}
and
\begin{equation}\label{eqn:psi_relation}
\ext{\psi}(\tvar) = 
\begin{cases}
\psi(\tvar) \quad \text{if} \quad \tvar\in \Y,\\
-\cst_\psi \quad \text{if} \quad \tvar = \ext{\tvar}.
\end{cases}
\end{equation}
\end{thm}

\begin{proof}
Let $\tm\in\PM(\Y)$. The optimality conditions for Problem \ref{prob:primal} given in Theorem \ref{thm:duality_existence_uniqueness} state that
\begin{align*}
\begin{rcases}
\bm\in\argmin \pf_\tm\\
\psi\in\argmax \df_\tm
\end{rcases}
&\iff
\begin{cases}
\bm = \bm_\psi\\
\psi \text{ is a Kantorovich potential from } \im_\bm \text{ to }\tm .
\end{cases}
\end{align*}
Combining this with the classical optimality conditions for optimal transport relating optimal transport plans to Kantorovich potentials (see for example \cite[Proposition 8]{merigot2021optimal}), and definition of $\bm_\psi$, gives
\begin{align}\label{eqn:equiv0}
\begin{rcases}
\bm\in\argmin \pf_\tm\\
\psi\in\argmax \df_\tm\\
\gamma \in \underset{\tilde{\gamma}\in\Gamma(\im_\bm,\tm)}{\argmin}\displaystyle{\underset{\X\times\Y}{\int}}c\, \rd\tilde{\gamma}
\end{rcases}
&\iff
\begin{cases}
\bm = \underset{\tvar\in\Y}{\max} \, \bmalt(\cdot,\tvar,\psi(\tvar) + \cst_\psi)\\
\psi^c \oplus \psi = c \quad \gamma\text{-a.e}
\end{cases}
\end{align}
where $\cst_\psi$ is the unique constant such that $\bm_\psi$ is a probability measure on $\base$. Conversely, again by \cite[Proposition 8]{merigot2021optimal},
\begin{align*}
\begin{rcases}
\bm = \underset{\tvar\in\Y}{\max}\, \bmalt(\cdot,\tvar,\ext{\psi}(\tvar)-\ext{\psi}(\ext{\tvar}))\\
\ext{\psi} \text{ is a Kantorovich potential from } \ext{\im} \text{ to }\ext{\tm}\\
\ext{\gamma} \in \underset{\tilde{\gamma}\in\Gamma(\ext{\im},\ext{\tm})}{\argmin}\displaystyle{\underset{\X\times\Y}{\int}}\ext{c}\, \rd\tilde{\gamma}
\end{rcases}
\iff
\begin{cases}
\bm = \underset{\tvar\in\Y}{\max} \, \bmalt(\cdot,\tvar,\ext{\psi}(\tvar) -\ext{\psi}(\ext{\tvar}))\\
\ext{\psi}^{\ext{c}} \oplus \ext{\psi} = \ext{c} \quad \ext{\gamma}\text{-a.e}.
\end{cases}
\end{align*}
Therefore, to establish \eqref{eqn:reduction}, it is sufficient to show that
\begin{align}\label{eqn:ext_support_cond}
\psi^c \oplus \psi = c \quad \gamma\text{-a.e}
\iff
\ext{\psi}^{\ext{c}} \oplus \ext{\psi} = \ext{c} \quad \ext{\gamma}\text{-a.e}
\end{align}
where $\psi\in\cts{\Y}$, $\cst_\psi\in\R$, and $\ext{\psi}\in \cts{\Y\times \{\ext{\tvar}\}}$ are related via \eqref{eqn:psi_relation},
and $\ext{\gamma}$ is defined by \eqref{eqn:gamma_ext}.
In particular, since $\base\times\{0\}$ is $\leb{d}$-negligible, the set $N \coloneqq (\base\times\{0\})\times\Y$ has zero $\gamma$ measure and zero $\ext{\gamma}$ measure, and it is sufficient to show that
\begin{equation}\label{eqn:suff_opt_cond}
\psi^c \oplus \psi = c \quad \text{on}\; \spt(\gamma)\setminus N
\iff
\ext{\psi}^{\ext{c}} \oplus \ext{\psi} = \ext{c} \quad \text{on}\; \spt(\ext{\gamma})\setminus N.
\end{equation}

First, we show that
\begin{equation}\label{eqn:ext_opt}
\ext{\psi}^{\ext{c}}(\svar) + \ext{\psi}(\tvar) = \ext{c}(\svar,\tvar) \quad \forall \, (\svar,\tvar)\in\spt(\ext{\gamma})\setminus (N\cup\spt(\gamma)),
\end{equation}
regardless of any hypotheses on $\psi$ and $\gamma$.
Let , $(\svar,\tvar)=((\shz,\svt),\tvar)\in\spt(\ext{\gamma})\setminus (N\cup\spt(\gamma))$. By construction 
$\svt\geq \bm(\shz)>0$, $\tvar=\ext{\tvar}$, and $\cst_\psi = -\ext{\psi}(\tvar) = \ext{c}(\svar,\tvar) - \ext{\psi}(\tvar)$.
Let $\tvara\in\Y$. On the one hand, if $c((\shz,0),\tvara)\leq \psi(\tvara) + \cst_\psi$, since $c((\shz,\cdot),\tvara)$ is increasing,
\[
c((\shz,\svt),\tvara)
\geq c((\shz,\bm(\shz)),\tvara)
\geq \psi(\tvara) + \cst_\psi,
\]
where the final inequality holds by definition of $\bmalt$.
On the other hand, if $c((\shz,0),\tvara)\geq \psi(\tvara) + \cst_\psi$, then
\[
c((\shz,\svt),\tvara)
\geq c((\shz,0),\tvara)
\geq \psi(\tvara) + \cst_\psi.
\]
In either scenario,
\[
\ext{c}(\svar,\tvara)-\psi(\tvara) \geq \cst_\psi = \ext{c}(\svar,\tvar) - \ext{\psi}(\tvar),
\]
so \eqref{eqn:ext_opt} holds.

To complete the proof, let $(\svar,\tvar)=((\shz,\svt),\tvar)\in \spt(\gamma)\setminus N$.
Then $\tvar\in \Y$ and $0< \svt \leq \bm(\shz)$. In particular, there exists $\tvara\in\Y$ such that
\[
c((\shz,\bm(\shz),\tvara) - \psi(\tvara) \leq \cst_\psi,
\]
otherwise $\bm(\shz)$ would be zero by definition.
Since $c((\shz,\cdot),\tvara)$ is increasing (Assumption \ref{ass:cost}.\ref{ass:cost_inc}),
\[
\psi^c(\svar)
= \psi^c((\shz,\svt)) 
\leq c((\shz,\svt),\tvara) - \psi(\tvara)
\leq c((\shz,\bm(\shz)),\tvara) - \psi(\tvara)\leq \cst_\psi.
\]
By construction
\[
\cst_\psi
= \ext{c}((\shz,\bm(\shz)),\ext{\tvar}) - \ext{\psi}(\ext{\tvar}).
\]
It follows that $\psi^c(\svar) = \ext{\psi}^{\ext{c}}(\svar)$.
Therefore,
\begin{equation}\label{eqn:eqiv_on_Y}
\psi^c(\svar) + \psi(\tvar)
= c(\svar,\tvar)
\iff
\ext{\psi}^{\ext{c}}(\svar) + \ext{\psi}(\tvar)
= \ext{c}(\svar,\tvar).
\end{equation}
Since $\spt(\gamma)\subseteq \spt(\ext{\gamma})$, this establishes the backwards implication in \eqref{eqn:suff_opt_cond}. Together, \eqref{eqn:ext_opt} and \eqref{eqn:eqiv_on_Y} establish the forwards implication in \eqref{eqn:suff_opt_cond}, which completes the proof.
\end{proof}

If Assumption \ref{ass:ext_twist} holds, then the cost function $\ext{c}$ is twisted. The following corollary is therefore a direct consequence of Theorem \ref{thm:reduction} and the Gangbo-McCann Theorem (see for example \cite{gangbo1996geometry} or \cite[Theorem 12]{merigot2021optimal}).

\begin{cor}\label{cor:reduction_map}
Let $\ext{c}$, $\ext{\im}$, $\ext{\tm}$, $\ext{\psi}$, $\ext{\gamma}$, and $P$ be defined as in Theorem \ref{thm:reduction}. If Assumption \ref{ass:ext_twist} holds, then there exists a unique optimal transport map $\ext{T}$ from $\ext{\im}$ to $\ext{\tm}$ for the cost $\ext{c}$, and it satisfies $\ext{\gamma} = (\mathrm{id}_{\base\times [0,P]},\ext{T})\#\ext{\im}$, and
\[
\nabla_{\svar} \ext{c}(\svar,\ext{T}(\svar)) = \nabla \ext{\psi}^c(\svar) \qquad \forall \, \svar\in \base\times [0,P].
\]
In particular, if $\bm\in \argmin \pf_\tm$ is the unique solution of Problem \ref{prob:primal} then $T\coloneqq \ext{T}\vert_{\X_\bm}$ is the optimal transport map from $\im_\bm$ to $\tm$ for the cost $c$, where we recall that $\X_{\bm} = \{(\shz,\svt)\in\X\, \vert \, \svt\leq \bm(\shz)\}$.
\end{cor}

With a view to solving Problem \ref{prob:primal} numerically, we now write down the formulation of Corollary \ref{cor:reduction_map} in the case that $\tm$ is a finitely supported measure. This reduces Problem \ref{prob:primal} to a standard semi-discrete optimal transport problem. (See \cite[Section 4]{merigot2021optimal} for an overview of semi-discrete optimal transport). By \cite[Proposition 37]{merigot2021optimal} the corresponding Monge problem has a solution $T$ expressed in \eqref{eqn:OT_map} in terms of so-called \emph{Laguerre cells} $L_i$ (see \eqref{eqn:Laguerre} and more generally \cite[Definition 17]{merigot2021optimal}). The intersection of any two Laguerre cells is $\leb{d}$-negligible \cite[Proposition 37]{merigot2021optimal}, so $T$ is well defined and corresponds to a labelled partition of the set $\X_\bm$.

\begin{cor}\label{cor:reduction_semi_discrete}
Suppose Assumption \ref{ass:ext_twist} holds.
Let $n\in \N$, and for $i\in \{1,\ldots,n\}$ let $\tvar_i\in \Y$ be distinct and $m_i>0$ be such that $\sum_{i=1}^n m_i = 1$. Let
\[
\tm = \sum_{i=1}^n m_i \delta_{\tvar_i},
\]
and let $\bm\in \argmin \pf_\tm$ be the unique solution of Problem \ref{prob:primal}. 
Let $\ext{\tvar}$, $\ext{c}$, $\ext{\im}$, and $\ext{\tm}$ be defined as in Theorem \ref{thm:reduction}, and let $\ext{\psi}$ be a Kantorovich potential between $\ext{\im}$ and $\ext{\tm}$ for the cost $\ext{c}$. 
Define 
$\Psi \coloneqq (\ext{\psi}(\tvar_i))_{i=1}^{n+1}$ and $\mathbf{\tvar} = (\tvar_i)_{i=1}^{n+1}$,
with the convention that $\tvar_{n+1} = \ext{\tvar}$. For $i\in \{1,\ldots,n+1\}$, define
\begin{equation}\label{eqn:Laguerre}
L_i(\Psi,\mathbf{\tvar})\coloneqq \left\{\svar\in \base\times [0,P]\, \bigg\vert \, \ext{c}(\svar,\tvar_i) - \Psi_i \leq \ext{c}(\svar,\tvar_j) - \Psi_j \quad \forall \, j\in \{1,\ldots,n+1\}\right\}.
\end{equation}
Then
\[
\mathrm{graph}(\bm) \subseteq \partial L_{n+1}(\Psi,\mathbf{\tvar}),
\]
and $T:\X_p\to\Y$ defined by
\begin{equation}\label{eqn:OT_map}
T \coloneqq \sum_{i=1}^n \tvar_i \chfun_{L_i(\Psi,\mathbf{y})}
\end{equation}
is the optimal transport map from $\im_\bm$ to $\tm$ for the cost $c$.
\end{cor}

\section{Conclusion}

In this work, we have given a rigorous definition of a Modified Lagrangian Mean (MLM) state in a single hemisphere in latitude-pressure coordinates (Definition \ref{def:MLM_state}). Such a state consists of a surface pressure function $\overline{p}$ determining the extent of the source domain, and an admissible transport map $T=(\ZAM,\PT)$ between the uniform density on the source domain and a mass distribution $\tm$ in the space of zonal angular momentum and potential temperature. We have proved that if $\tm$ has compact support in $(0,+\infty)^2$, then there exists a unique energy-minimising MLM state defined on a domain excluding the equator and the pole (Theorem \ref{thm:duality_existence_uniqueness}). Optimality conditions imply that this state is in fact the solution of a single optimal transport problem (Theorem \ref{thm:reduction}). Our results hold not only for the cost function defining the energy of an MLM state \eqref{eqn:bgs_cost} but for a wide class of cost functions satisfying mild assumptions. These include the cost function \eqref{eqn:SG_cost} defining free-surface variants of the semi-geostrophic equations \cite{cheng2016semigeostrophic,cullen2014solutions,cullen2001variational}.

Regarding energy-minimising MLM states specifically, our results come with three main caveats. First, the exclusion of the extreme latitudes in the given hemisphere. The pole is excluded to ensure that the cost function is sufficiently regular (locally Lipschitz; Assumption \ref{ass:cost}.\ref{ass:cost_cts}). The equator is excluded to ensure that the cost function is twisted (Assumption \ref{ass:twist}) and admits a twisted extension (Assumption \ref{ass:ext_twist}). This guarantees the existence of optimal transport maps (rather than merely optimal transport plans), which can be interpreted as energy-minimising rearrangements of zonal angular momentum and potential temperature. Second, and for the same two reasons, an arbitrary minimum pressure is assumed and built into the cost function \eqref{eqn:bgs_cost}. Third, the zonal angular momentum is assumed to be positive and bounded away from zero to ensure that the cost function is twisted.

In addition to the general existence and uniqueness result, we have shown that to compute an energy-minimising MLM state from a discrete mass distribution in zonal angular momentum and potential temperature coordinates (target space), as in the setting of \cite{methven2015slowly}, it is sufficient to solve a semi-discrete optimal transport problem with a twisted cost (Corollary \ref{cor:reduction_semi_discrete}). This reduces to maximising the Kantorovich dual functional, which is a concave, finite dimensional, unconstrainted maximisation problem \cite[Section 4]{merigot2021optimal}, and is numerically tractable \edit{using, for example, the damped Newton method of \cite{kitagawa2019convergence}. Figure \ref{fig:computed_tessellation} illustrates an example where the discrete mass distribution in target space is specified from integrals calculated from 3-D global atmospheric data and the semi-discrete optimal transport solution obtains the corresponding distribution of mass in the source space corresponding to a meridional cross-section in latitude-pressure coordinates.}

\section*{Acknowledgements}
\edit{The authors thank Thilo D. Stier (University of G\"{o}ttingen) for implementing the numerical semi-discrete optimal transport solver that was used to generate Figure \ref{fig:computed_tessellation}.}
The authors would like to thank the Isaac Newton Institute for Mathematical Sciences, Cambridge, for support and hospitality during the satellite programme `Geophysical Fluid Dynamics; from Mathematical Theory to Operational Prediction', hosted at the University of Reading, where work on this paper was initiated, and which was supported by the UK Engineering and Physical Sciences Research Council (EPSRC) grant EP/R014604/1. CE gratefully acknowledges the support of: the EPSRC via the grant EP/W522570/1; the German Research Foundation (DFG) via the CRC 1456 `Mathematics of Experiment', project A03; and the University of Reading Advancing the Frontiers of Earth System Prediction (AFESP) Programme (award number A3720200). DPB acknowledges financial support from the EPSRC grant EP/V00204X/1.

\section*{Data Availability}
\edit{The ECMWF reanalysis data used to produce Figure \ref{fig:computed_tessellation} is available from the ECMWF MARS archive (https://confluence.ecmwf.int/display/CKB/The+family+of+ERA5+datasets). Note that the full resolution native ECMWF model output in terms of spectral coefficients on terrain-following model levels was used to compute the mass and circulation integrals (rather than the gridded, pressure level data available at the Copernicus Data Store). Here, ERA-Interim data was used, but ERA5 data gives very similar results for the mass distribution $\nu$ in $(z,\theta)$ coordinates.}

\bibliographystyle{siam}
\bibliography{bibliography.bib}

\end{document}